\documentclass[leqno, amstex,12pt, amssymb, final]{article}

\usepackage{mathtext}
\usepackage[cp1251]{inputenc}
\usepackage[T2A]{fontenc}
\usepackage[dvips]{graphicx}
\usepackage{amsmath}
\usepackage{amssymb}
\usepackage{amsxtra}
\usepackage{latexsym}
\usepackage{ifthen}
\date{}

\newcounter{lemma}

\newcounter{corollary}

\newcounter{remark}

\newcounter{theorem}

\newcounter{proposition}

\newcounter{example}

\begin{document}

\markboth{D.~ROMASH, E.~SEVOST'YANOV}{\centerline{On uniformly
lightness ...}}

\def\cc{\setcounter{equation}{0}
\setcounter{figure}{0}\setcounter{table}{0}}

\overfullrule=0pt


\title{{\bf On uniformly lightness of ring mappings and convergence to a homeomorphism}}

\author{D.~ROMASH$^1$\footnote{dromash@num8erz.eu}, E.~SEVOST'YANOV$^{1,
2}$\footnote{Corresponding author, esevostyanov2009@gmail.com}}

\maketitle

{\small $^{1}$Zhytomyr Ivan Franko State University, 40, Velyka
Berdychivs'ka Str., 10 008  Zhytomyr, UKRAINE

$^{2}$Institute of Applied Mathematics and Mechanics of NAS of
Ukraine, 19 Henerala Batyuka Str., 84 116 Slov'yans'k, UKRAINE


\medskip

\maketitle

\medskip
{\bf Abstract.} A family of mappings is called uniformly light if
the image of the continuum under these mappings cannot be contracted
to a point under the sequence of mappings of the family. In this
paper, we are interested in the problem of the uniform lightness of
a family of homeomorphisms satisfying upper moduli inequalities. We
have shown that a family of such homeomorphisms satisfies the
above-mentioned condition of uniform lightness if the majorant
participating in the modulus estimate defining the family is
integrable over almost all spheres. Under the same conditions, we
show that this family of homeomorphisms is uniformly open, i.e.,
their image contains a ball of fixed radius, independent of each
mapping separately. As an application of the results obtained, we
have proved the assertion about the uniform convergence of
homeomorphisms to a homeomorphism.


\medskip
\centerline{\bf 1. Introduction}

\medskip
In the theory of quasiconformal mappings, theorems containing
properties of the limit mapping in some class are of great
importance. In particular, the following statement is true (see,
e.g., \cite[Theorems~21.9, 21.11]{Va}).

\medskip
{\bf Theorem~A.} {\it Let $D$ be a domain in ${\Bbb R}^n,$
$n\geqslant 2,$ and let $f_j:D\rightarrow {\Bbb R}^n,$ $j=1,2,\ldots
,$ be a sequence of $K_0$-quasiconformal mappings, $1\leqslant
K_0<\infty,$ that converges to some mapping $f:D\rightarrow
\overline{{\Bbb R}^n}$ locally uniformly in $D$ by the metric $h.$
Then $f$ is either a homeomorphism $f:D\rightarrow {\Bbb R}^n,$ or a
constant $c\in\overline{{\Bbb R}^n}.$}

\medskip
Similar statements have been repeatedly proven by different authors
in other formulations, including classes of mappings, more general
than quasiconformal. In particular, the following statement is true
(see, e.g., \cite[Theorem~4.2]{RS}, cf.~\cite[Theorem~1]{Cr}).

\medskip
{\bf Theorem~B.} {\it Let $D$ be a domain in ${\Bbb R}^n,$ $n\geqslant
2,$ and let $Q:D \rightarrow (0,\infty)$ be a measurable function
such that
\begin{equation}\label{eq2AA}
\int\limits_{0}^{\varepsilon(x_0)}\frac{dr}{rq_{x_0}^{\frac{1}{n-1}}(r)}=\infty\quad
\quad \forall\,x_0\in D
\end{equation}
for a positive $\varepsilon(x_0)< {\rm dist}\, (x_0,
\partial D)$ where $q_{x_0}(r)$ denotes the average of
$Q(x)$ over the sphere $|x-x_0|=r.$ Suppose that $f_m,$
$m=1,2,\ldots,$ is a sequence of ring $Q$-ho\-me\-o\-mor\-phisms
from $D$ into ${\Bbb R}^n$ converging locally uniformly to a mapping
$f.$ Then the mapping $f$ is either a constant in $\overline{{\Bbb
R}^n}$ or a homeomorphism into ${\Bbb R}^n.$ }

\medskip
The definition of ring $Q$-homeomorphisms used above may be found
in~\cite{MRSY} (it is also given below in the text as one of the key
definitions of our manuscript). Some analogue of Theorem~B was also
proved under the condition $Q\in L^1(D)$ for some (narrower than
ring $Q$-homeomorphisms) class of strong ring $Q$-homeomorphisms
(see Theorem~7.7 in~\cite{MRSY}). The main purpose of this paper is
to show that the result of Theorem~B is valid under condition~$Q\in
L^1(D).$ Such a result has not been obtained anywhere so far, and
its proof requires the development of a technique slightly different
from that which has already been used. We should also note that in
our recent paper~\cite{IRS} a certain result on this topic was
obtained, but only in a special case where the mapped domains have a
special geometry. The purpose of the present paper is to establish
this result without any restrictions on the mapped domains.

\medskip
Before the study the convergence problems, we investigate some other
problems, which have independent interest and play a role in proving
the main result of the work. First, we show that under some not very
strong restrictions the family of homeomorphisms under the
investigation is uniformly light, i.e., the image of the continuum
under such mappings cannot decrease infinitely. We also show that
the considered family of mappings is uniformly open, i.e., the image
of a ball under this family contains a ball of fixed radius. These
studies contain the ideas and approaches developed by us in
works~\cite{IRS} and~\cite{ST}, however, the research technique used
here is much broader and deeper.

\medskip
Let $x_0\in\overline{D},$ $x_0\ne\infty,$
$$
B(x_0, r)=\{x\in {\Bbb R}^n: |x-x_0|<r\}\,, \quad {\Bbb B}^n:=B(0,
1)\,,$$
\begin{equation}\label{eq1ED}
S(x_0,r) = \{ x\,\in\,{\Bbb R}^n : |x-x_0|=r\}\,,\end{equation}
$${\Bbb S}^{n-1}:=S(0, 1)\,,$$
\begin{equation}\label{eq1**A} A=A(x_0, r_1, r_2)=\{ x\,\in\,{\Bbb R}^n :
r_1<|x-x_0|<r_2\}\,.
\end{equation}
In what follows, $\Omega_n$ denotes the volume of the unit ball
${\Bbb B}^n$ in ${\Bbb R}^n,$ and $\omega_{n-1}$ denotes the area of
the unit sphere ${\Bbb S}^{n-1}$ in ${\Bbb R}^n.$

\medskip
Let $S_i=S(x_0, r_i),$ $i=1,2,$ where spheres $S(x_0, r_i)$ centered
at $x_0$ of the radius $r_i$ are defined in~(\ref{eq1ED}). Let
$Q:{\Bbb R}^n\rightarrow {\Bbb R}^n$ be a Lebesgue measurable
function satisfying the condition $Q(x)\equiv 0$ for $x\in{\Bbb
R}^n\setminus D.$ A mapping $f:D\rightarrow \overline{{\Bbb R}^n}$
is called a {\it ring $Q$-mapping at the point $x_0\in
\overline{D}\setminus \{\infty\}$}, if the condition
\begin{equation} \label{eq2*!A}
M(f(\Gamma(S_1, S_2, D)))\leqslant \int\limits_{A\cap D} Q(x)\cdot
\eta^n (|x-x_0|)\, dm(x)
\end{equation}
holds for all $0<r_1<r_2<d_0:=\sup\limits_{x\in D}|x-x_0|$ and all
Lebesgue measurable functions $\eta:(r_1, r_2)\rightarrow [0,
\infty]$ such that
\begin{equation}\label{eq8BC}
\int\limits_{r_1}^{r_2}\eta(r)\,dr\geqslant 1\,.
\end{equation}
A mapping $f$ is called a {\it ring $Q$-mapping in $D,$} if
condition~(\ref{eq2*!A}) is satisfied at every point $x_0\in D,$ and
a {\it ring $Q$-mapping in $\overline{D},$} if the
condition~(\ref{eq2*!A}) holds at every point $x_0\in\overline{D}.$
For the properties of such mappings see, e.g., \cite{MRSY}. Let $h$
be a chordal metric in $\overline{{\Bbb R}^n},$
$$
h(x,\infty)=\frac{1}{\sqrt{1+{|x|}^2}}\,,\quad
h(x,y)=\frac{|x-y|}{\sqrt{1+{|x|}^2} \sqrt{1+{|y|}^2}}\qquad x\ne
\infty\ne y\,,
$$
and let $h(E):=\sup\limits_{x,y\in E}\,h(x,y)$ be a chordal diameter
of a set~$E\subset \overline{{\Bbb R}^n}.$ We also use the relations
$h(A ,B)=\inf\limits_{x\in A, y\in B}h(x, y),$
$d(E):=\sup\limits_{x,y\in E}\,|x-y|$ and $d(A ,B)=\inf\limits_{x\in
A, y\in B}|x-y|.$

\medskip
An analogue of Theorem~\ref{th4} given below was obtained earlier in
\cite[Theorem~4.2]{RS} for similar classes of mappings under some
more stringent conditions on the function $Q,$ cf. Theorems~A and B,
Theorem~7.7 in~\cite{MRSY}, \cite[Corollary~21.3]{Va} and
\cite[Theorem~1]{Cr}.

\medskip
\begin{theorem}\label{th4}
{\it\, Let $D$ be a domain in ${\Bbb R}^n,$ $n\geqslant 2,$ let
$Q:D\rightarrow[0, \infty]$ be a Lebesgue measurable function and
let $f_m:D\rightarrow \overline{{\Bbb R}^n},$ $m=1,2,\ldots, $ be a
sequence of homeomorphisms satisfying the relations (\ref{eq2*!A})
and (\ref{eq8BC}) for all $x_0\in D$ and all $0<r_1<r_2<r_0:={\rm
dist\,}(x_0, \partial D).$ Let $Q\in L_{\rm loc}^1(D).$ If $f_m$
converges to $f:D\rightarrow\overline{{\Bbb R}^n}$ locally
uniformly, then $f$ either a homeomorphism $f:D\rightarrow
\overline{{\Bbb R}^n},$ or a constant $c\in \overline{{\Bbb R}^n}.$

If, in addition, $f_m(x)\ne \infty$ for all $x\in D,$ then $f$
either a homeomorphism $f:D\rightarrow {\Bbb R}^n,$ or a constant
$c\in \overline{{\Bbb R}^n}.$ }
\end{theorem}

\medskip
Let $D$ be a domain in ${\Bbb R}^n,$ $n\geqslant 2,$ and let $K$ be
a compact subset of $D.$ A family $\frak{F}$ of mappings
$f:D\rightarrow\overline{{\Bbb R}^n}$ will be called {\it uniformly
light on $K,$} if the following condition holds: given
$\varepsilon>0$ there is $\delta_1(\varepsilon, K)>0$ such that
$h(f(C))\geqslant \delta_1$ for any $f\in \frak{F}$ and any
continuum $C\subset K$ with $h(C)\geqslant \varepsilon.$ A family
$\frak{F}$ of mappings $f:D\rightarrow\overline{{\Bbb R}^n}$ will be
called {\it equi-uniform on $K,$} if the following condition holds:
given $\varepsilon>0$ there is $\delta_1(\varepsilon, K)>0$ such
that $h(f(x), f(y))\geqslant \delta_1$ for any $f\in \frak{F}$ and
any $x, y\in K$ with $h(x, y)\geqslant \varepsilon.$ Obviously, if a
family is equi-uniform on $K,$ then it also is uniformly light
on~$K.$

\medskip
Given $x_0\in\overline{{\Bbb R}^n}$ and $r>0$ we set $B_h(x_0,
r):=\{x\in \overline{{\Bbb R}^n}: h(x, x_0)<r\}.$ A family
$\frak{F}$ of mappings $f:D\rightarrow\overline{{\Bbb R}^n}$ is
called {\it uniformly open at a point $x_0\in D,$} if for every
$\varepsilon_0>0,$ $\varepsilon_0<{\rm dist}\,(x_0, \partial D),$
there is $r_0=r_0(x_0, \varepsilon_0)>0$ such that $B_h(f(x_0),
r_0)\subset f(B(x_0, \varepsilon_0))$ for every $f\in \frak{F}.$ A
family $\frak{F}$ of mappings $f:D\rightarrow\overline{{\Bbb R}^n}$
is called {\it uniformly open on $K,$} if for every
$\varepsilon_0>0$ there exists $r_0=r_0(K, \varepsilon_0)>0$ such
that $B_h(f(x_0), r_0)\subset f(B(x_0, \varepsilon_0))$ for every
$f\in \frak{F}$ and every $B(x_0, \varepsilon_0)\in K.$

\medskip
Given a domain $D$ in ${\Bbb R}^n,$ $n\geqslant 2,$ a Lebesgue
measurable function $Q:D\rightarrow [0, \infty],$ a compact set
$K\subset D$ and a number $\delta>0$ denote by
$\frak{F}^{\delta}_{K, Q}(D)$ a family of homeomorphisms
$f:D\rightarrow \overline{{\Bbb R}^n}$ satisfying the
relations~(\ref{eq2*!A})--(\ref{eq8BC}) at any point $x_0\in D$ for
all $0<r_1<r_2<d_0:=\sup\limits_{x\in D}|x-x_0|$ such that $h(f(K),
\partial f(D))=\inf\limits_{x\in f(K), y\in \partial f(D)}
h(x, y)\geqslant \delta.$ The following results hold.

\medskip
\begin{theorem}\label{th1}
{\it\,Let $D$ be a domain in ${\Bbb R}^n,$ $n\geqslant 2.$ Assume
that, for each point $x_0\in D$ and for every
$0<r_1<r_2<r_0:=\sup\limits_{x\in D}|x-x_0|$ there is a set
$E_1\subset[r_1, r_2]$ of a positive linear Lebesgue measure such
that the function $Q$ is integrable with respect to
$\mathcal{H}^{n-1}$ over the spheres $S(x_0, r)$ for every $r\in
E_1.$

\medskip
1) If $D$ is bounded, then the family $\frak{F}^{\delta}_{K, Q}(D)$
is uniformly light on $K;$ moreover, $\frak{F}^{\delta}_{K, Q}(D)$
is equi-uniform on $K.$

\medskip
2) if $Q\in L^1(D)$ and $f(x)\ne\infty$ for every $x\in D,$ then
there exist constants $C, C_1>0$ such that
\begin{equation}\label{eq1D}
|f(x)-f(y)|\geqslant C_1\cdot\exp\left\{-\frac{2\Vert
Q\Vert_1}{C|x-y|^n}\right\}
\end{equation}
for all $x, y\in K$ and every $f\in \frak{R}^{\delta}_{K, Q}(D).$}
\end{theorem}

\medskip
\begin{corollary}\label{cor1}
{\it\, The statement of Theorem~\ref{th1} remains true if, instead
of the condition regarding the integrability of the function $Q$
over spheres with respect to some set $E_1$ is replaced by a simpler
condition: $Q\in L^1(D).$}
\end{corollary}

\medskip
\begin{theorem}\label{th2}
{\it\, Let $D$ be a bounded domain. Assume that, for each point
$x_0\in D$ and for every $0<r_1<r_2<r_0:=\sup\limits_{x\in
D}|x-x_0|$ there is a set $E_1\subset[r_1, r_2]$ of a positive
linear Lebesgue measure such that the function $Q$ is integrable
with respect to $\mathcal{H}^{n-1}$ over the spheres $S(x_0, r)$ for
every $r\in E_1.$ Then the family $\frak{F}^{\delta}_{K, Q}(D)$ is
uniformly open on $K.$} \end{theorem}

\medskip
\begin{corollary}\label{cor2}
{\it\, The statement of Theorem~\ref{th2} remains true if, instead
of the condition regarding the integrability of the function $Q$
over spheres with respect to some set $E_1$ is replaced by a simpler
condition: $Q\in L^1(D).$}
\end{corollary}

\medskip
Condition~(\ref{eq1D}) was obtained in~\cite{IRS} under some
additional condition, namely, we assumed that all mapped domains
$f(D)$ must be quasiextremal distance domains by Gehring-Martio with
some general ``quasiextremal'' constant $A\geqslant 1,$
cf.~\cite{GM}. Observe that, the condition $h(f(K),
\partial f(D))\geqslant \delta$ in the definition of the class
$\frak{F}^{\delta}_{K, Q}(D)$ is somewhat difficult to verify, so
below we will consider a slightly different condition.

\medskip
Given a domain $D$ in ${\Bbb R}^n,$ $n\geqslant 2,$ a non-degenerate
continuum $A\subset D,$ a Lebesgue measurable function
$Q:D\rightarrow [0, \infty]$ and numbers $\delta, r>0$ we denote by
$\frak{A}^{\delta, r}_{A, Q}(D)$ a family of all homeomorphisms
$f:D\rightarrow \overline{{\Bbb R}^n}$
satisfying~(\ref{eq2*!A})--(\ref{eq8BC}) for all
$0<r_1<r_2<d_0:=\sup\limits_{x\in D}|x-x_0|$ such that
$h(f(A))\geqslant \delta$ and satisfying the condition:
$h(E)\geqslant r$ whenever $E$ is any component of $\partial f(D).$
The following result holds.

\medskip
\begin{theorem}\label{th1A}
{\it\,Assume that, for each point $x_0\in D$ and for every
$0<r_1<r_2<r_0:=\sup\limits_{x\in D}|x-x_0|$ there is a set
$E_1\subset[r_1, r_2]$ of a positive linear Lebesgue measure such
that the function $Q$ is integrable with respect to
$\mathcal{H}^{n-1}$ over the spheres $S(x_0, r)$ for every $r\in
E_1.$

\medskip
1) If $D$ is bounded, then the family $\frak{A}^{\delta, r}_{A,
Q}(D)$ is equi-uniform (uniformly light) and uniformly open on any
compactum $K\subset D;$

\medskip
2) if $Q\in L^1(D)$ and $f(x)\ne\infty$ for every $x\in D,$ then for
any $K\subset D$ there exist constants $C, C_1>0$ such that
\begin{equation}\label{eq1DA}
|f(x)-f(y)|\geqslant C_1\cdot\exp\left\{-\frac{\Vert
Q\Vert_1}{C|x-y|^n}\right\}
\end{equation}
for all $x, y\in K$ and every $f\in \frak{A}^{\delta, r}_{A, Q}(D).$
}
\end{theorem}

\medskip
\begin{corollary}\label{cor3}
{\it\, The statement of Theorem~\ref{th1} remains true if, instead
of the condition regarding the integrability of the function $Q$
over spheres with respect to some set $E_1$ is replaced by a simpler
condition: $Q\in L^1(D).$}
\end{corollary}

\medskip
As one of the most important consequences of Theorem~\ref{th1A}, we
have the following statement.

\medskip
\begin{theorem}\label{th6}
{\it Let $D$ be a bounded domain in ${\Bbb R}^n,$ $n\geqslant 2,$
let $Q:D\rightarrow[0, \infty]$ be a Lebesgue measurable function
and let $f_m:D\rightarrow \overline{{\Bbb R}^n},$ $m=1,2,\ldots, $
be a sequence of homeomorphisms satisfying the relations
(\ref{eq2*!A}) and (\ref{eq8BC}) for all $x_0\in D$ and all
$0<r_1<r_2<r_0:={\rm dist\,}(x_0, \partial D).$ Assume that, for
each point $x_0\in D$ and for every $0<r_1<r_2<r_0:=r_0:={\rm
dist\,}(x_0, \partial D)$ there is a set $E_1\subset[r_1, r_2]$ of a
positive linear Lebesgue measure such that the function $Q$ is
integrable with respect to $\mathcal{H}^{n-1}$ over the spheres
$S(x_0, r)$ for every $r\in E_1.$

If $f_m$ converges to $f:D\rightarrow\overline{{\Bbb R}^n}$ locally
uniformly, then $f$ either a homeomorphism $f:D\rightarrow
\overline{{\Bbb R}^n},$ or a constant $c\in \overline{{\Bbb R}^n}.$

If, in addition, $f_m(x)\ne \infty$ for all $x\in D,$ then $f$
either a homeomorphism $f:D\rightarrow {\Bbb R}^n,$ or a constant
$c\in \overline{{\Bbb R}^n}.$ }
\end{theorem}

\begin{remark}\label{rem1}
In particular, the conclusion of Theorem~\ref{th1A} holds, if
instead of above conditions on $Q,$ we require that $Q\in L_{\rm
loc}^1(D).$
\end{remark}

\medskip
\centerline{\bf 2. Proof of Theorem~\ref{th1}}

\medskip
The following statement holds, see, e.g.,
\cite[Theorem~1.I.5.46]{Ku$_2$}).

\medskip
\begin{proposition}\label{pr2}
{\it\, Let $A$ be a set in a topological space $X.$ If the set $C$
is connected and $C\cap A\ne \varnothing\ne C\setminus A,$ then
$C\cap
\partial A\ne\varnothing.$}
\end{proposition}

\medskip The following statement can be found
in~\cite[Lemma~4.2]{Vu}.

\medskip
\begin{proposition}\label{pr1}
{\it\, Let $D$ be an open half space or an open ball in ${\Bbb R}^n$
and let $E$ and $F$ be subsets of $\overline{D}.$ Then
$$M(\Gamma(E, F, \overline{D}))
\geqslant \frac{1}{2}\cdot M(\Gamma(E, F, \overline{{\Bbb R}^n}))\,.$$
}
\end{proposition}
Let $(X, \mu)$ be a metric space with measure $\mu$ and of Hausdorff
dimension $n.$ For each real number $n\geqslant 1,$ we define {\it
the Loewner function} $\Phi_n:(0, \infty)\rightarrow [0, \infty)$ on
$X$ as
\begin{equation}\label{eq2H}
\Phi_n(t)=\inf\{M_n(\Gamma(E, F, X)): \Delta(E, F)\leqslant t\}\,,
\end{equation}
where the infimum is taken over all disjoint nondegenerate continua
$E$ and $F$ in $X$ and
$$\Delta(E, F):=\frac{{\rm dist}\,(E,
F)}{\min\{d(E), d(F)\}}\,.$$
A pathwise connected metric measure space $(X, \mu)$ is said to be a
{\it Loewner space} of exponent $n,$ or an $n$-Loewner space, if the
Loewner function $\Phi_n(t)$ is positive for all $t> 0$ (see
\cite[Section~2.5]{MRSY} or \cite[Ch.~8]{He}). Observe that, ${\Bbb
R}^n$ is a Loewner space (see \cite[Theorem~8.2]{He}).

\medskip
Let $X$ be a Loewner space in which the relation $\mu(B_R)\leqslant
C^{\,*}R^n$ holds for some constant $C^{\,*}\geqslant 1,$ for some
exponent $n>0$ and for all closed balls $B_R$ of radius $R>0.$ Let
$\Phi_n(t)$ be a Loewner function in~(\ref{eq2H}), corresponding to
the space $X.$ Then, by virtue of~\cite[Theorem~8.23]{He}, there is
$\delta_0>0$ and some constant $C>0$ such that
\begin{equation}\label{eq3C}
\Phi_n(t)\geqslant C\log\frac{1}{t}
\end{equation}
for $0<t<\delta_0.$

\medskip
The following statement holds, see e.g. \cite[Theorem~7.2]{MRSY}.

\medskip
\begin{proposition}\label{pr3}{\it\,
Let $D$ be a domain in ${\Bbb R}^n,$ $n\geqslant 2,$ and
$Q:D\rightarrow[0,\infty]$ a measurable function. A homeomorphism
$f:D\rightarrow {\Bbb R}^n$ is a ring $Q-$homeomorphism at a point
$x_0$ if and only if for every $0<r_1<r_2< d_0={\rm dist}\,
(x_0,\partial D),$
$$M(\Gamma (f(S_1),f(S_2),f(D)))\leqslant
\frac{\omega_{n-1}}{I^{n-1}}$$
where $\omega_{n-1}$ is the area of the unit sphere in ${\Bbb R}^n,$
$q_{x_0}(r)$ is the mean value of $Q$ over the sphere $|x-x_0|=r,$
$S_j=\{ x\in{\Bbb R}^n : |x-x_0|=r_j\},$ $j=1,2,$  and
$I=I(r_1,r_2)=\int\limits_{r_1}^{r_2}
\frac{dr}{rq_{x_0}^{\frac{1}{n-1}}(r)}.$}\end{proposition}

\medskip
\begin{lemma}\label{lem1}{\it\,Let $D$ be a domain in ${\Bbb R}^n,$
$n\geqslant 2,$ let $\delta>0$ and let $K$ be a compactum in $D.$
Let $x, y \in K,$ $x\ne y,$ and let $f\in\frak{F}^{\delta}_{K,
Q}(D).$ Assume that $f(x)\ne\infty\ne f(y)$ and
$$|f(x)-f(y)|<\min\left\{\frac{\delta_0\delta}{2},
\frac{\delta}{4}\right\}\,,$$
where $\delta_0$ is a number from~(\ref{eq3C}). Then
$\overline{B(f(x), \delta/2)}\subset f(D).$

\medskip
In addition, let $r=r(t)=y+(x-y)t,$ $-\infty<t<\infty.$ Then there
exists $\kappa_1\geqslant 1$ such that $z^{(1)}=y+(x-y)\kappa_1$ and
\begin{equation}\label{eq14***}
C\cdot\log \frac{\delta}{2|f(x)-f(y)|}\leqslant 2\cdot
M(f(\Gamma(S(z^{(1)}, \varepsilon^{(1)}), S(z^{(1)},
\varepsilon^{(2)}), A(z^{(1)}, \varepsilon^{(1)},
\varepsilon^{(2)}))))\,,
\end{equation}
where $C$ is a constant in~(\ref{eq3C}),
$\varepsilon^{(1)}:=|z^{(1)}-x|$ and
$\varepsilon^{(2)}:=|z^{(1)}-y|.$}
\end{lemma}

\medskip
\begin{proof}
We fix $x, y\in K\subset D,$ $x\ne y,$ and $f\in \frak{F}^{A,
\delta}_{K, Q}(D).$ Draw through the points $x$ and $y$ a straight
line $r=r(t)=y+(x-y)t,$ $-\infty<t<\infty$ (see Figure~\ref{fig2}).
\begin{figure}[h]
\centerline{\includegraphics[scale=0.4]{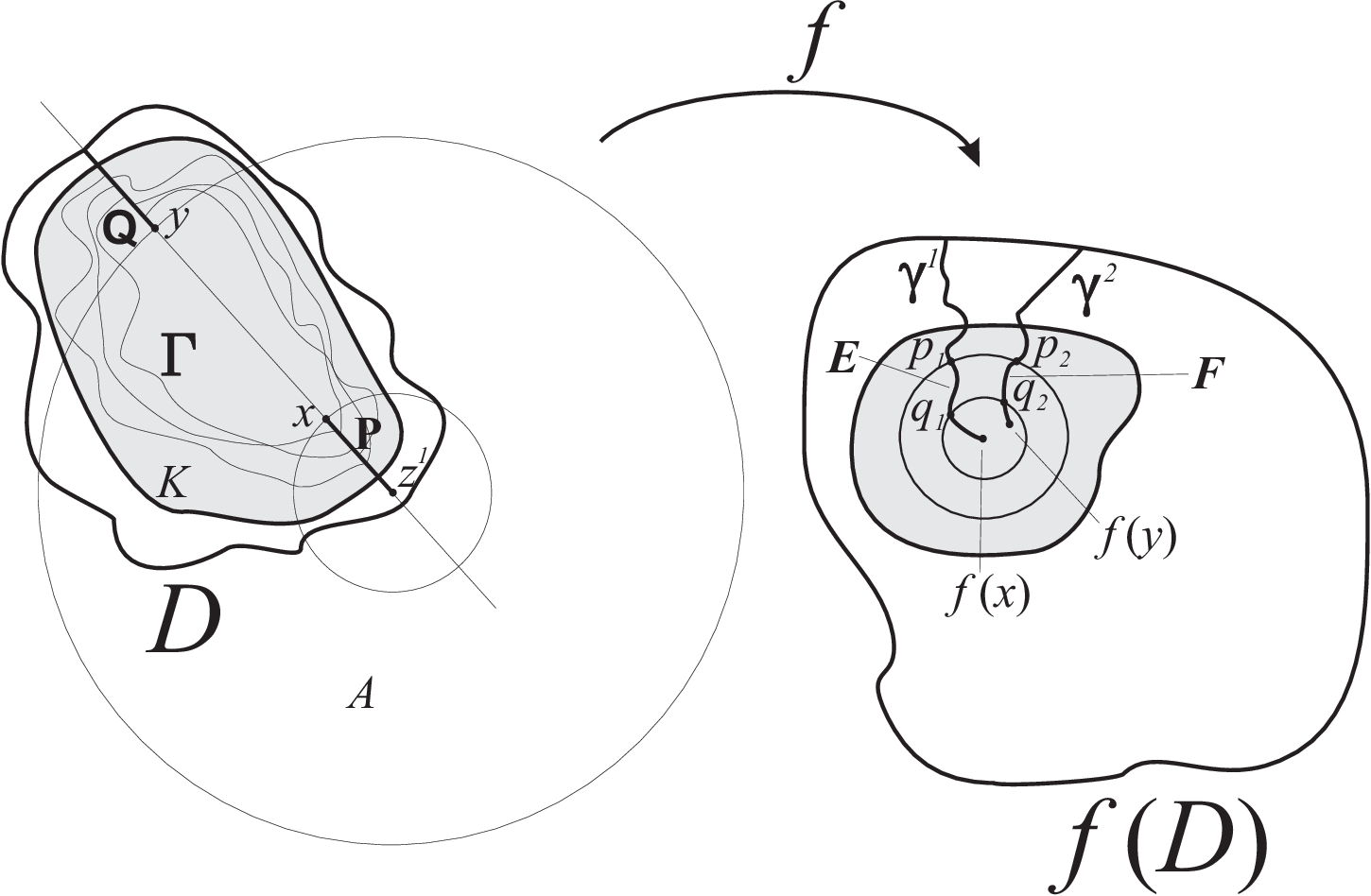}} \caption{To
proof of Theorem~\ref{th1}}\label{fig2}
\end{figure}
Let $t_1=\sup\limits_{t\geqslant 1: r(t)\in D}t$ and
$t_2:=\inf\limits_{t\leqslant 0: r(t)\in D}t.$ If $r(t)\in D$ for
any $t\geqslant 1,$ we set $t_1=+\infty.$ Similarly, if $r(t)\in D$
for any $t\leqslant 0,$ we set $t_2=-\infty.$ Let $P:=r|_{[1, t_1)}$
and $Q:=r|_{(t_2, 0]}.$ Let $\gamma^1=f(P)$ and $\gamma^2=f(Q).$ Due
to Proposition~13.5 in \cite{MRSY} homeomorphisms preserve the
boundary of a domain. Thus, $\gamma^1(t)\rightarrow \partial f(D)$
as $t\nearrow t_1$ and $\gamma^2(t)\rightarrow \partial f(D)$ as
$t\searrow t_2.$ Observe that the ball $\overline{B(f(x),
\delta/2)}$ lies inside $f(D).$ Indeed, otherwise $\overline{B(f(x),
\delta/2)}\cap f(D)\ne\varnothing\ne \overline{B(f(x),
\delta/2)}\setminus f(D)$ and, consequently, $\overline{B(f(x),
\delta/2)}\cap \partial f(D)\ne\varnothing$ by
Proposition~\ref{pr2}. Now, let $y\in \overline{B(f(x),
\delta/2)}\cap
\partial f(D).$ Now, $d(f(K), \partial f(D))\leqslant
|f(x)-y|\leqslant \delta/2<\delta$ that contradicts the definition
of the class $\frak{F}^{A, \delta}_{K, Q}(D).$

\medskip
Observe that $|\gamma^1|\cap B(f(x), \delta/4)\ne\varnothing\ne
|\gamma^1|\setminus B(f(x), \delta/4)$ and $|\gamma^2|\cap B(f(x),
\delta/4)\ne\varnothing\ne |\gamma^1|\setminus B(f(x), \delta/4).$
By Proposition~\ref{pr2} $|\gamma^1|\cap S(f(x),
\delta/4)\ne\varnothing\ne |\gamma^2|\cap S(f(x), \delta/4).$
Similarly, $|\gamma^1|\cap S(f(x), \delta/2)\ne\varnothing\ne
|\gamma^2|\cap S(f(x), \delta/2).$ Let $q_1=\gamma^1(\vartheta_1)\in
S(f(x), \delta/4)$ and $q_2=\gamma^2(\vartheta_2)\in S(f(x),
\delta/4)$ and let $p_1=\gamma^1(\kappa_1)\in S(f(x), \delta/4)$ and
$p_2=\gamma^2(\kappa_2)\in S(f(x), \delta/4)$ for some
$\kappa_1\geqslant \vartheta_1\geqslant 1$ and $\kappa_2\leqslant
\vartheta_2\leqslant 0.$ Then $E:=\gamma^1|_{[1, \kappa_1]}$ and
$F:=\gamma^2|_{[\kappa_2, 0]}$ are continua in $f(D)$ such that
$d(E)\geqslant d(p_1, q_1)=\delta/4$ and $d(F)\geqslant d(p_2,
q_2)=\delta/4.$ Without loss of generality we may assume that $E$
and $F$ belong to $\overline{B(f(x), \delta/2)}.$

\medskip
Next, we should use the fact that the $n$-dimensional Euclidean
space is a Loewner space (see, for example, \cite[Theorem~8.2]{He}).
Observe that,
%
%
$$\Delta(E, F):=\frac{{\rm dist}\,(E, F)}{\min\{d(E), d(F)\}}
\leqslant \frac{2|f(x)-f(y)|}{\delta}$$
%
because ${\rm dist}\,(E, F)\leqslant |f(x)-f(y)|.$
Then, by the definition of the Loewner function $\Phi_n(t)$
in~(\ref{eq2H}) we obtain that
\begin{equation}\label{eq4C}
\Phi_n\left(\frac{2|f(x)-f(y)|}{\delta}\right)\leqslant M(\Gamma(E,
F, {\Bbb R}^n))\,.
\end{equation}
Now, by~(\ref{eq3C})
\begin{equation}\label{eq5C}
C\cdot\log \frac{\delta}{2|f(x)-f(y)|}\leqslant
\Phi_n\left(\frac{2|f(x)-f(y)|}{\delta}\right)\,.
\end{equation}
Now, by~(\ref{eq4C}) and~(\ref{eq5C}) we obtain that
\begin{equation}\label{eq6C}
C\cdot\log \frac{\delta}{2|f(x)-f(y)|}\leqslant M(\Gamma(E, F, {\Bbb
R}^n))\,.
\end{equation}
Since by the proving above $\overline{B(f(x), \delta/2)}\subset
f(D),$ we obtain that $$\Gamma(E, F, \overline{B(f(x),
\delta/2)})\subset \Gamma(E, F, f(D))\,.$$ Now, by
Proposition~\ref{pr1},
\begin{equation}\label{eq1AA}
M(\Gamma(E, F, \overline{{\Bbb R}^n}))= M(\Gamma(E, F, {\Bbb
R}^n))\leqslant 2 M(\Gamma(E, F, \overline{B(f(x),
\delta/2)}))\leqslant 2M(\Gamma(E, F, f(D)))\,.
\end{equation}
Now, by~(\ref{eq6C}) and~(\ref{eq1AA}) we obtain that
\begin{equation}\label{eq1AAA}
C\cdot\log \frac{\delta}{2|f(x)-f(y)|}\leqslant 2M(\Gamma(E, F,
f(D)))\,.
\end{equation}
Now let us prove the upper bound for $M(\Gamma(E, F, f(D))).$
Indeed, by the definition,
\begin{equation}\label{eq2}
\Gamma(E, F, f(D))=f(\Gamma(P_1, Q_1, D))\,,
\end{equation}
where $P_1:=f^{\,-1}(E)$ and $Q_1:=f^{\,-1}(F).$ In addition, let us
to show that
\begin{equation}\label{eq1A}
\Gamma(P_1, Q_1, D)>\Gamma(S(z^{(1)}, \varepsilon^{(1)}), S(z^{(1)},
\varepsilon^{(2)}), A(z^{(1)}, \varepsilon^{(1)},
\varepsilon^{(2)}))\,,
\end{equation}
where $z^{(1)}:=f^{\,-1}(p_1),$ $\varepsilon^{(1)}:=|z^{(1)}-x|$ and
$\varepsilon^{(2)}:=|z^{(1)}-y|.$ Observe that
$$|y-x|+\varepsilon^{(1)}=$$
\begin{equation}\label{eq5B}
=|y-x|+|x-z^{(1)}|= |z^{(1)}-y|=\varepsilon^{(2)}\,,
\end{equation}
and, thus, $\varepsilon^{(1)}<\varepsilon^{(2)}.$ Let
$\gamma\in\Gamma(P_1, Q_1, D).$ Then $\gamma:[0, 1]\rightarrow {\Bbb
R}^n,$ $\gamma(0)\in P_1,$ $\gamma(1)\in Q_1$ and $\gamma(s)\in D$
for $0<s<1.$ Recall that,
$$z^{(1)}=y+(x-y)\kappa_1$$
for $\kappa_1\geqslant 1.$ Since $\gamma(0)\in P_1,$ there is
$1\leqslant t\leqslant \kappa_1$ such that $\gamma(0)=y+(x-y)t.$
Thus,
$$|\gamma(0)-z^{(1)}|=|(x-y)(\kappa_1-t)|$$
\begin{equation}\label{eq2B}
\leqslant |x-y|(\kappa_1-1)=|(x-y)\kappa_1+y-x|\end{equation}
$$=|x-z^{(1)}|=\varepsilon^{(1)}\,.$$
On the other hand, since $\gamma(1)\in Q_1,$ there is $p\leqslant 0$
such that
$$\gamma(1)=y+(x-y)p\,.$$ In this case, we obtain that
$$
|\gamma(1)-z^{(1)}|=|(x-y)(\kappa_1-p)|$$
\begin{equation}\label{eq3A}
\geqslant |(x-y)\kappa_1|=|(x-y)\kappa_1 +y-y|
\end{equation}
$$=|y-z^{(1)}|=\varepsilon^{(2)}\,.$$
Since $\varepsilon^{(1)}<\varepsilon^{(2)},$ due to~(\ref{eq3A}) we
obtain that
\begin{equation}\label{eq4E}
|\gamma(1)-z^{(1)}|>\varepsilon^{(1)}\,.
\end{equation}
It follows from~(\ref{eq2B}) and (\ref{eq4E}) that $|\gamma|\cap
\overline{B(z^{(1)}, \varepsilon^{(1)})}\ne\varnothing\ne
(D\setminus \overline{B(z^{(1)}, \varepsilon^{(1)})})\cap|\gamma|.$
In this case, by Proposition~\ref{pr2} there is $\zeta_1\in (0, 1)$
such that~$\gamma(\zeta_1)\in S(z^{(1)}, \varepsilon^{(1)}).$ We may
assume that $\gamma(t)\not\in B(z^{(1)}, \varepsilon^{(1)})$ for
$t\in (\zeta_1, 1).$ Put $\alpha^1:=\gamma|_{[\zeta_1, 1]}.$

\medskip
On the other hand, since $\varepsilon^{(1)}<\varepsilon^{(2)}$ and
$\gamma(\zeta_1)\in S(z^{(1)}, \varepsilon^{(1)}),$ we obtain that
$|\alpha^1|\cap B(z^{(1)}, \varepsilon^{(2)})\ne\varnothing.$
By~(\ref{eq3A}) we obtain that $(D\setminus B(z^{(1)},
\varepsilon^{(2)}))\cap |\alpha^1|\ne\varnothing.$ Thus, again by
Proposition~\ref{pr2} there is $\zeta_2\in (\zeta_1, 1)$ such that
$\alpha^1(\zeta_2)\in S(z^{(1)}, \varepsilon^{(2)}).$ We may assume
that $\gamma(t)\in B(z^{(1)}, \varepsilon^{(2)})$ for $t\in
(\zeta_1, \zeta_2).$ Set $\alpha^2:=\alpha^1|_{[\zeta_1, \zeta_2]}.$
Now, $\gamma>\alpha^2$ and $\alpha^2\in \Gamma(S(z^{(1)},
\varepsilon^{(1)}), S(z^{(1)}, \varepsilon^{(2)}), A).$ Thus,
(\ref{eq1A}) is proved.

\medskip
By the inequalities~(\ref{eq1AAA}), (\ref{eq2}) and (\ref{eq1A}), by
Proposition~\ref{pr3} and by the definition of the class
$\frak{F}^{A, \delta}_{K, Q}(D),$ we obtain that
$$
C\cdot\log \frac{\delta}{2|f(x)-f(y)|}\leqslant 2\cdot
M(f(\Gamma(S(z^{(1)}, \varepsilon^{(1)}), S(z^{(1)},
\varepsilon^{(2)}), A(z^{(1)}, \varepsilon^{(1)},
\varepsilon^{(2)}))))\,.
$$
The latter relation completes the proof of Lemma~\ref{lem1}.~$\Box$
\end{proof}

{\it Proof of Theorem~\ref{th1}.} The proof of Theorem~\ref{th1}
uses some approaches applied in~\cite{SevSkv$_2$} for some another
class of mappings, cf.~\cite{IRS}. Firstly we prove that
$\frak{F}^{\delta}_{K, Q}(D)$ is equi-uniform on $K.$ (From this, in
particular, it will follow that $\frak{F}^{\delta}_{K, Q}(D)$
uniformly light).

\medskip
Let us to prove Theorem~\ref{th1} by the contradiction, i.e., assume
that the statement of this theorem is false. Then there exists
$\varepsilon_0>0$ such that, for any $m\in {\Bbb N}$ there is a
continuum $x_m, y_m\in K$ and $f_m\in \frak{F}^{A, \delta}_{K,
Q}(D)$ such that $h(x_m, y_m)\geqslant \varepsilon_0,$ however,
$h(f_m(x_m), f_m(y_m))<1/m.$ Now, $|x_m-y_m|\geqslant \varepsilon_0$
for any $m\in {\Bbb N},$ as well. In addition, since $h(f_m(x_m),
f_m(y_m))<1/m$ for $m=1,2,\ldots,$ we may consider that $f_m(x_m),
f_m(y_m)\rightarrow \omega_*$ as $m\rightarrow \infty$ for some
$\omega_*\in\overline{{\Bbb R}^n}.$ Using the inversion
$\psi(x)=\frac{x}{|x|^2}$ and considering the mappings
$\widetilde{f}_m:=\psi\circ f_m,$ if required, we may consider that
$\omega_*\ne \infty$ and $f_m(x_m)\ne\infty\ne f_m(y_m)$ for any
$m\in {\Bbb N}.$ Now, since the conformal modulus $M$ is an
invariant under conformal mapping $\psi$ (see e.g.
\cite[Theorem~8.1]{Va}), $\widetilde{f}_m$
satisfy~(\ref{eq2*!A})--(\ref{eq8BC}) as well, while the condition
$h(\widetilde{f}_m(K),
\partial \widetilde{f}_m(D))\geqslant \widetilde{\delta}$ holds with
some another $\widetilde{\delta}>0.$

\medskip
Applying Lemma~\ref{lem1}, we set in~(\ref{eq14***}) $x:=x_m,$
$y:=y_m$ and $f:=f_m.$ Now we obtain that
$$C\cdot\log \frac{\delta}{2|f_m(x_m)-f_m(y_m)|}\leqslant$$
\begin{equation}\label{eq1BA}
\leqslant 2 \cdot M(f(\Gamma(S(z^{(1, m)}, \varepsilon^{(1, m)}),
S(z^{(1, m)}, \varepsilon^{(2, m)}), A(z^{(1, m)}, \varepsilon^{(1,
m)}, \varepsilon^{(2, m)}))))\,,
\end{equation}
where $z^{(1, m)}:=f^{\,-1}(p_m),$ $\varepsilon^{(1, m)}:=|z^{(1,
m)}-x_m|$ and $\varepsilon^{(2, m)}:=|z^{(1, m)}-y_m|,$ $p_m\in
f_m(D).$ Since $D$ is bounded, we may consider that $x_m\rightarrow
x_0\in \overline{D},$ $y_m\rightarrow y_0\in \overline{D}$ and
$z^{(1, m)}\rightarrow z_*$ as $m\rightarrow\infty.$ Now,
$\varepsilon^{(1, m)}\rightarrow \varepsilon_1$ and
$\varepsilon^{(2, m)}\rightarrow \varepsilon_2$ as
$m\rightarrow\infty$ for some $\varepsilon_1, \varepsilon_2>0.$

\medskip
Fix
\begin{equation}\label{eq7A}
0<\varepsilon<\varepsilon_0/2\,.
\end{equation}
Let $x^{\,*}\in D$ be such that $|x^{\,*}-z_*|<\varepsilon/3.$
Additionally, let $m_0\in{\Bbb N}$ be such that
\begin{equation}\label{eq8A}
|z^{(1, m)}-z_*|<\varepsilon/3\,,\quad
\varepsilon_1-\varepsilon/3<\varepsilon^{(1,
m)}<\varepsilon_1+\varepsilon/3\quad\forall\,\, m>m_0\,.
\end{equation}

Let $x\in B(z^{(1, m)}, \varepsilon^{(1, m)}).$ Now, by~(\ref{eq8A})
and the triangle inequality we obtain that
$$|x-x^{\,*}|\leqslant |x-z^{(1, m)}|+|z^{(1, m)}-z_*|+
|z_*-x^{\,*}|$$
\begin{equation}\label{eq5}
< \varepsilon^{(1, m)}+\varepsilon/3+\varepsilon/3<
\varepsilon_1+\varepsilon\,,\quad m>m_0\,.
\end{equation}
It follows from~(\ref{eq5}) that
\begin{equation}\label{eq6A}
B(z^{(1, m)}, \varepsilon^{(1, m)})\subset B(x^{\,*},
\varepsilon_1+\varepsilon)\,, \quad m>m_0\,.
\end{equation}

Let $R_0$ be such that
\begin{equation}\label{eq6B}
\varepsilon_1+\varepsilon<R_0<\varepsilon_1+\varepsilon_0-\varepsilon\,.
\end{equation}
Note that the choice of the number $R_0$ in the formula~(\ref{eq6B})
is possible by the definition of the number $\varepsilon$
in~(\ref{eq7A}). Let $y\in B(x^{\,*}, R_0).$ Again, by the triangle
inequality and relations~(\ref{eq8A}), (\ref{eq6B}) we obtain that
$$|y-z^{(1, m)}|\leqslant
|y-x^{\,*}|+|x^{\,*}-z_*|+|z_*-z^{(1, m)}|$$
\begin{equation}\label{eq5a}
<R_0+2\varepsilon/3<\varepsilon_1+\varepsilon_0-\varepsilon+2\varepsilon/3=
\varepsilon_1+\varepsilon_0-\varepsilon/3<\varepsilon^{(1,
m)}+\varepsilon_0\leqslant \varepsilon^{(2, m)}\,.
\end{equation}

It follows from~(\ref{eq5a}) that
\begin{equation}\label{eq9A}
B(x^{\,*}, R_0)\subset B(z^{(1, m)},  \varepsilon^{(2, m)})\,,\quad
m>m_0\,.
\end{equation}
Recall that, a family of paths $\Gamma_1$ in ${\Bbb R}^n$ is said to
be {\it minorized} by a family of paths $\Gamma_2$ in ${\Bbb R}^n,$
abbr. $\Gamma_1>\Gamma_2,$ if, for every path $\gamma_1\in\Gamma_1$,
there is a path $\gamma_2\in\Gamma_2$ such that $\gamma_2$ is a
restriction of $\gamma_1.$ In this case,
\begin{equation}\label{eq32*A}
\Gamma_1
> \Gamma_2 \quad \Rightarrow \quad M(\Gamma_1)\leqslant M(\Gamma_2)
\end{equation} (see~\cite[Theorem~1]{Fu}).
Taking into account relations~(\ref{eq1BA}), (\ref{eq9A})
and~(\ref{eq32*A}) as well as Propositions~\ref{pr2} and~\ref{pr3},
we obtain that
$$C\cdot\log \frac{\delta}{2|f_m(x_m)-f_m(y_m)|}\leqslant$$
\begin{equation}\label{eq10}
\leqslant 2\cdot M(f(\Gamma(S(z^{(1, m)}, \varepsilon^{(1, m)}),
S(z^{(1, m)}, \varepsilon^{(2, m)}), A(z^{(1, m)}, \varepsilon^{(1,
m)}, \varepsilon^{(2, m)}))))\leqslant
\end{equation}
$$\leqslant 2\cdot (f(\Gamma(S(x^{\,*}, \varepsilon_1+\varepsilon,
S(x^{\,*}, R_0), A(x^{\,*}, \varepsilon_1+\varepsilon,
R_0))))\leqslant \frac{2\omega_{n-1}}{I^{n-1}}<\infty\,,\quad m>
m_0\,,$$
where
$$I=I(x^{\,*}, \varepsilon_1+\varepsilon,
R_0)=\int\limits_{\varepsilon_1+\varepsilon}^{R_0}
\frac{dr}{rq_{x^{\,*}}^{\frac{1}{n-1}}(r)}$$
and $I<\infty$ because, by the assumptions of the theorem, $Q$ is
integrable on $S(x^{\,*}, r)$ for $r\in E_1$ for some set $E_1$ of
positive linear Lebesgue measure in $[\varepsilon_1+\varepsilon,
R_0].$ It follows from~(\ref{eq10}) that
\begin{equation}\label{eq2D}
|f_m(x_m)-f_m(y_m)|\geqslant
\frac{\delta}{2}\exp\left\{-\frac{2\omega_{n-1}}{CI^{n-1}}\right\}=const>0\,,\quad
m>m_0\,.
\end{equation}
The relation~(\ref{eq2D}) contradicts the assumption that
$|f_m(x_m)-f_m(y_m)|<1/m,$ $m=1,2,\ldots.$ The obtained
contradiction proves the theorem in the case when $Q$ is integrable
on family of spheres $S(x_0, r),$ $r\in E_1.$

\medskip
It remains to consider the case $Q\in L^1(D).$ Indeed, in the
notions of Lemma~\ref{lem1}, by this Lemma, we obtain that
\begin{equation}\label{eq15}
C\cdot\log \frac{\delta}{2|f(x)-f(y)|}\leqslant 2\cdot
M(f(\Gamma(S(z^{(1)}, \varepsilon^{(1)}), S(z^{(1)},
\varepsilon^{(2)}), A(z^{(1)}, \varepsilon^{(1)},
\varepsilon^{(2)}))))
\end{equation}
for every $x, y \in K,$ $x\ne y,$ and every
$f\in\frak{F}^{\delta}_{K, Q}(D),$ where $C$ is a constant
in~(\ref{eq3C}), $\varepsilon^{(1)}:=|z^{(1)}-x|$ and
$\varepsilon^{(2)}:=|z^{(1)}-y|.$

Now, we set
$$\eta(t)= \left\{
\begin{array}{rr}
\frac{1}{|x-y|}, & t\in [\varepsilon^{(1)}, \varepsilon^{(2)}],\\
0,  &  t\not\in [\varepsilon^{(1)}, \varepsilon^{(2)}]\,.
\end{array}
\right. $$
Observe that, $\eta$ satisfies the relation~(\ref{eq8BC}) for
$r_1=\varepsilon^{(1)}$ and $r_2=\varepsilon^{(2)}.$ Indeed, it
follows from~(\ref{eq5B}) that
$$r_1-r_2=\varepsilon^{(2)}-\varepsilon^{(1)}=|y-z^{(1)}|-|x-
z^{(1)}|=$$$$=|x-y|\,.$$
Then
$\int\limits_{\varepsilon^{(1)}}^{\varepsilon^{(2)}}\eta(t)\,dt=(1/|x-y|)\cdot
(\varepsilon^{(2)}-\varepsilon^{(1)})= 1.$ By~(\ref{eq15}) and by
the definition of the class $\frak{F}^{\delta}_{K, Q}(D),$ we obtain
that
\begin{equation}\label{eq14***A}
C\cdot\log \frac{\delta}{2|f(x)-f(y)|}\leqslant
\frac{2}{|x-y|^n}\int\limits_{D} Q(w)\,dm(w)=\frac{\Vert
Q\Vert_1}{{|x-y|}^n}\,.
\end{equation}
It follows from~(\ref{eq14***A}) that the relation~(\ref{eq1D})
holds with $C_1=\frac{\delta}{2}$ whenever
$|f(x)-f(y)|<\frac{\delta_0\delta}{2},$ where $\delta_0$ is a number
from~(\ref{eq3C}). Let now $|f(x)-f(y)|\geqslant
\frac{\delta_0\delta}{2}.$ Then
$$\exp\left\{-\frac{2\Vert Q\Vert_1}{C|x-y|^{n}}\right\}\leqslant
\exp\left\{-\frac{2\Vert Q\Vert_1}{C(d(K))^{n}}\right\}=
\exp\left\{-\frac{2\Vert Q\Vert_1}{C(d(K))^{n}}\right\}\cdot
\frac{\delta_0\delta}{2}\cdot \frac{2}{\delta_0\delta}\leqslant$$
$$\leqslant \exp\left\{-\frac{2\Vert Q\Vert_1}{C(d(K))^{n}}\right\}
\cdot \frac{2}{\delta_0\delta}|f(x)-f(y)|\,,$$
or, equivalently,
$$|f(x)-f(y)|\geqslant \frac{\delta_0\delta}{2}
\cdot\exp\left\{2\frac{\Vert
Q\Vert_1}{C(d(K))^{n}}\right\}\exp\left\{-\frac{2\Vert
Q\Vert_1}{C|x-y|^{n}}\right\}\,.$$
Finally, the relation~(\ref{eq1D}) holds for
$C_1:=\min\left\{\frac{\delta}{2},  \frac{\delta_0\delta}{2}
\cdot\exp\left\{\frac{2\Vert
Q\Vert_1}{C(d(K))^{n}}\right\}\right\}.$~$\Box$

\medskip
{\it Proof of Corollary~\ref{cor1}.} Let $0<r_0:=\sup\limits_{y\in
D}|x-x_0|.$ We may assume that $Q$ is extended by zero outside $D.$
Let $Q\in L^1(D).$ By the Fubini theorem (see, e.g.,
\cite[Theorem~8.1.III]{Sa}) we obtain that
$$\int\limits_{r_1<|x-x_0|<r_2}Q(x)\,dm(x)=\int\limits_{r_1}^{r_2}
\int\limits_{S(x_0, r)}Q(x)\,d\mathcal{H}^{n-1}(x)dr<\infty\,.$$
This means the fulfillment of the condition of the integrability of
the function $Q$ on the spheres with respect to any subset $E_1$ in
$[r_1, r_2].$~$\Box$}

\medskip
{\it Proof of Theorem~\ref{th2}.} Assume the contrary, i.e., the
family $\frak{F}^{\delta}_{K, Q}(D)$ is not uniformly open on $K.$
Then there exists $\varepsilon_0>0$ such that for any $m\in {\Bbb
N}$ there exists $x_m\in K$ and $f_m\in \frak{F}^{\delta}_{K, Q}(D)$
such that $B(x_m, \varepsilon_0)\subset K$ and $B_h(f_m(x_m),
1/m)\setminus f_m(B(x_m, \varepsilon_0))\ne\varnothing.$ Let $y_m\in
B_h(f_m(x_m), 1/m)\setminus f_m(B(x_m, \varepsilon_0)).$ We may
consider that $f_m(x_m)$ and $y_m$ converge to some point
$\omega_{*}$ as $m\rightarrow \infty.$ We may consider
$\omega_{*}\ne\infty,$ otherwise we consider
$\widetilde{f}_m:=\psi\circ f_m,$ $\psi(x)=\frac{x}{|x|^2},$ instead
$f_m$ follow.

\medskip
Now, $f_m(x_m)\ne\infty\ne y_m$ for sufficiently large $m\in {\Bbb
N}.$ Join the points $f_m(x_m)$ and $y_m$ by the segment
$r_m(t)=f_m(x_m)+t(y_m-f_m(x_m)),$ $t\in [0, 1].$ Since $|r_m|\cap
f_m(B(x_m, \varepsilon_0))\ne\varnothing\ne |r_m|\setminus
f_m(B(x_m, \varepsilon_0)),$ by Proposition~\ref{pr2} there is a
point $z_m=r_m(t_m)\in \partial f_m(B(x_m, \varepsilon_0).$ Without
loss of generality, we may assume that the path $\beta_m:=r_m|_{[0,
t_m)}$ lies in $f_m(B(x_m, \varepsilon_0)).$ Now, set
$\alpha_m=f_m^{\,-1}(\beta_m).$ Since homeomorphisms preserve the
boundary, $\alpha_m\rightarrow S(x_m, \varepsilon_0)$ (see
Proposition~13.5 in \cite{MRSY}). Now, there is a point $z_m\in
|\alpha_m|$ with $|z_m-x_m|>\varepsilon_0/2,$ so we may chose a
subcontinuum $E_m\subset|\alpha_m|$ with $d(E_m)\geqslant
\varepsilon_0/2.$ Now, since $E_m$ belongs to a bounded domain $D,$
we obtain that $h(E_m)\geqslant \varepsilon^{\,*}>0$ for any $m\in
{\Bbb N}$ and some $\varepsilon^{\,*}>0,$ as well. On the other
hand, $f_m(E_m)\subset |\beta_m|$ and
$d(\beta_m)=|f_m(x_m)-y_m|<1/m,$ $m\in {\Bbb N}.$ Since $f_m(x_m)$
and $y_m$ converge to some point $\omega_{*}$ as $m\rightarrow
\infty$ and $\omega_{*}\ne\infty,$ and follows that
$h(\beta_m)\rightarrow 0$ as $m\rightarrow\infty,$ as well.

\medskip
Finally, $h(E_m)\geqslant \varepsilon_0/2,$ $E_m\subset K$ for any
$m=1,2,\ldots $ and $h(f_m(E_m))\rightarrow 0$ as
$m\rightarrow\infty$ for $f_m\in\frak{F}^{\delta}_{K, Q}(D).$ The
latter contradicts the statement of Theorem~\ref{th1}. The obtained
contradiction proves the theorem.~$\Box$

\medskip
{\it Proof of Corollary~\ref{cor2}} is similar ro the proof of
Corollary~\ref{cor1}.~$\Box$

\medskip
\centerline{\bf 2. Lemmas on continua approaching the boundary.}

\medskip
The following statement holds.

\medskip
\begin{lemma}\label{lem3} {\bf(V\"{a}is\"{a}l\"{a}'s lemma on the weak flatness of inner points).}
{\sl\, Let $n\geqslant 2 $, let $D$ be a domain in $\overline{{\Bbb
R}^n},$ and let $x_0\in D.$ Then for each $P>0$ and each
neighborhood $U$ of point $x_0$ there is a neighborhood $V\subset U$
of the same point such that $M(\Gamma(E, F, D))> P$ for any continua
$E, F \subset D $ intersecting $\partial U$ and $\partial V.$}
\end{lemma}

\medskip
The proof of Lemma~\ref{lem3} is essentially given by
V\"{a}is\"{a}l\"{a} in~\cite[(10.11)]{Va}, however, we have also
given a formal proof, see \cite[Lemma~2.2]{SevSkv$_1$}.~$\Box$

\medskip
Several of our publications contain the assertion that for
homeomorphisms satisfying conditions~(\ref{eq2*!A})--(\ref{eq8BC}),
the image of the continuum cannot approach the boundary of the
mapped domain if this boundary is, say, weakly flat (see, for
example, \cite[Lemma~4.1]{SevSkv$_1$}). Now we will establish that
this fact is satisfied in a much more general situation, namely, no
geometry is required on the boundary of the mapped domain, and the
domain can even change under the mapping. The following assertion is
true.

\medskip
\begin{lemma}\label{lem2}{\it\, Let $D$ be a domain in ${\Bbb R}^n,$ $n\geqslant 2,$ let $A$
be a non-degenerate continuum in $D$ and let $f_m:D\rightarrow
\overline{{\Bbb R}^n}$ be a sequence of homeomorphisms in
$\frak{A}^{\delta, r}_{A, Q}(D)$ for some function $Q,$ continuum
$A\subset D$ and numbers $\delta, r>0.$ Assume that, for each point
$x_0\in D$ and for every $0<r_1<r_2<r_0:=\sup\limits_{x\in
D}|x-x_0|$ there is a set $E_1\subset[r_1, r_2]$ of a positive
linear Lebesgue measure such that the function $Q$ is integrable
with respect to $\mathcal{H}^{n-1}$ over the spheres $S(x_0, r)$ for
every $r\in E_1.$ Then there exists $\delta_1>0$ such that
$h(f_m(A),
\partial f_m(D))\geqslant \delta_1$ for any $m=1,2,\ldots .$ }
\end{lemma}

\medskip
\begin{proof}
Since $D$ is a domain in ${\Bbb R}^n,$ $\partial D\ne\varnothing.$
In addition, since $C(\partial D, f_m)\subset \partial f_m(D)$
whenever $f_m$ is a homeomorphism (see Proposition~13.5 in
\cite{MRSY}), $\partial f_m(D)\ne \varnothing,$ as well. Thus, the
quantity $h(f_m(A),
\partial f_m(D))$ is well-defined, so the formulation of the lemma is
correct.

\medskip
Let us prove Lemma~\ref{lem2} by the contradiction, partially using
the approach of the proof of Lemma~4.1 in \cite{SevSkv$_1$}. Suppose
that the conclusion of the lemma is not true. Then for each $k\in
{\Bbb N}$ there is some number $m=m_k$ such that $h(f_{m_k}(A),
\partial f_{m_k}(D))<1/k.$ In order not to complicate the notation, we will further assume that
$h(f_{m}(A),
\partial f_m(D))<1/m,$ $m=1,2,\ldots. $ Note that
the set $f_{m}(A)$ is compact as a continuous image of a compact set
$A\subset D$ under the mapping~$f_{m}.$ In this case, there are
elements $x_m\in f_{m}(A)$ and $y_m\in \partial f_m(D)$ such that
$h(f_{m}(A),
\partial f_m(D))=h(x_m, y_m)<1/m.$ Let $E_m$ be a component of $\partial
f_m(D)$ consisting $y_m.$ Due to the compactness of $\overline{{\Bbb
R}^n},$ we may assume that $y_m\rightarrow y_0$ as $m\rightarrow
\infty;$ then also $x_m\rightarrow y_0$ as $m\rightarrow \infty.$

\medskip
Put $P>0$ and $U=B_h(y_0, r_0)=\{y\in \overline{{\Bbb R}^n}: h(y,
y_0)<r_0\},$ where $2r_0:=\min\{r/2, \delta/2\},$ where $r$ and
$\delta$ are numbers from the condition of the lemma. Observe that
$E_m\cap U\ne\varnothing\ne E_m\setminus U$ for sufficiently large
$m\in{\Bbb N},$ since $y_m\rightarrow y_0$ as $m\rightarrow \infty,$
$y_m\in E_m;$ besides that $h(E_m)\geqslant r>r/2\geqslant 2r_0$ and
$h(U)\leqslant 2r_0.$ Since $E_m$ is a continuum, $E_m\cap
\partial U\ne\varnothing$ by Proposition~\ref{pr2}. Similarly,
$f_{m}(A)\cap U\ne\varnothing\ne f_{m}(A)\setminus U$ for
sufficiently large $m\in{\Bbb N},$ since $x_m\rightarrow y_0$ as
$m\rightarrow \infty,$ $x_m\in f_{m}(A);$ besides that
$h(f_{m}(A))\geqslant \delta>\delta/2\geqslant 2r_0$ and
$h(U)\leqslant 2r_0.$ Since $f_m(A)$ is a continuum, $f_m(A)\cap
\partial U\ne\varnothing$ by Proposition~\ref{pr2}.
By the proving above,
\begin{equation}\label{eq8}
E_m\cap
\partial U\ne\varnothing\ne f_m(A)\cap
\partial U
\end{equation}
for sufficiently large $m\in {\Bbb N}.$ By Lemma~\ref{lem3} there is
$V\subset U,$ $V$ is a neighborhood of $y_0,$ such that
\begin{equation}\label{eq9}
M(\Gamma(E, F, \overline{{\Bbb R}^n}))>P
\end{equation}
for any continua $E, F\subset \overline{{\Bbb R}^n}$ with $E\cap
\partial U\ne\varnothing\ne E\cap \partial V$ and $F\cap \partial
U\ne\varnothing\ne F\cap \partial V.$
Arguing similarly to above, we may prove that
$$
E_m\cap
\partial V\ne\varnothing\ne f_m(A)\cap
\partial V
$$
for sufficiently large $m\in {\Bbb N}.$ Thus, by~(\ref{eq9})
\begin{equation}\label{eq9B}
M(\Gamma(f_m(A), E_m,\overline{{\Bbb R}^n}))>P
\end{equation}
for sufficiently large $m=1,2,\ldots .$ We now prove that the
relation~(\ref{eq9B}) contradicts the definition of $f_m$
in~(\ref{eq2*!A})--(\ref{eq8BC}).

\medskip
Indeed, let $\gamma:[0, 1]\rightarrow \overline{{\Bbb R}^n}$ be a
path in $\Gamma(f_m(A), E_m, \overline{{\Bbb R}^n}),$ i.e.,
$\gamma(0)\in f_m(A),$ $\gamma(1)\in E_m$ and $\gamma(t)\in
\overline{{\Bbb R}^n}$ for $t\in (0, 1).$ Let
$t_m=\sup\limits_{\gamma(t)\in f_m(D)}t$ and let
$\alpha_m(t)=\gamma|_{[0, t_m)}.$ Let $\Gamma_m$ consists of all
such paths $\Gamma_m,$ now $\Gamma(f_m(A), E_m, \overline{{\Bbb
R}^n})>\Gamma_m$ and by the minorization principle of the modulus
(see~\cite[Theorem~1]{Fu})
\begin{equation}\label{eq12}
M(\Gamma_m)\geqslant M(\Gamma(f_m(A), E_m, \overline{{\Bbb
R}^n}))\,.
\end{equation}
Let $\Delta_m:=f^{\,-1}_m(\Gamma_m).$ Now
\begin{equation}\label{eq13}
M(f_m(\Delta_m))=M(\Gamma_m)\,.
\end{equation}
Observe that, $\beta_m\in \Delta_m$ if and only if
$\beta_m(t)=f^{\,-1}_m(\alpha_m(t))=f^{\,-1}_m(\gamma|_{[0, t_m)})$
for some $\gamma\in \Gamma(f_m(A), E_m, \overline{{\Bbb R}^n})$ and
$t_m=\sup\limits_{\gamma(t)\in f_m(D)}t.$ Since homeomorphisms
preserve the boundary, $\beta_m(t)\rightarrow \partial D$ as
$t\rightarrow t_m.$

Let $\varepsilon:={\rm dist}\,(K, \partial D)$ (if $D$ has no finite
boundary points, we may choose any positive number as $\varepsilon,$
say $\varepsilon=1$). We cover the continuum $A$ with balls $B(x,
\varepsilon/4),$ $x\in A.$ Since $A$ is a compact set, we may assume
that $A\subset \bigcup\limits_{i=1}^{M_0}B(x_i, \varepsilon/4),$
$x_i\in A,$ $i=1,2,\ldots, M_0,$ $1\leqslant M_0<\infty.$ By the
definition, $M_0$ depends only on $A,$ in particular, $M_0$ does not
depend on $m.$ We set
Note that
\begin{equation}\label{eq6D}
\Delta_m=\bigcup\limits_{i=1}^{M_0}\Gamma_{mi}\,,
\end{equation}
where $\Gamma_{mi}$ consists of all paths $\gamma:[0, 1)\rightarrow
D$ in $\Delta_m$ such that $\gamma(0)\in B(x_i, \varepsilon/4)$ and
$\gamma(t)\rightarrow \partial D$ as $t\rightarrow 1-0.$ We now show
that
\begin{equation}\label{eq7C}
\Gamma_{mi}>\Gamma(S(x_i, \varepsilon/4), S(x_i, \varepsilon/2),
A(x_i, \varepsilon/4, \varepsilon/2))\,.
\end{equation}
Indeed, let $\gamma\in \Gamma_{mi},$ in other words, $\gamma:[0,
1)\rightarrow D,$ $\gamma\in \Delta_m,$ $\gamma(0)\in B(x_i,
\varepsilon/4)$ and $\gamma(t)\rightarrow \partial D$ as
$t\rightarrow 1-0.$ Now, by the definition of $\varepsilon,$
$|\gamma|\cap B(x_i, \varepsilon/4)\ne\varnothing\ne |\gamma|\cap
(D\setminus B(x_i, \varepsilon/4)).$ Therefore, by
Proposition~\ref{pr2} there is $0<t_1<1$ such that $\gamma(t_1)\in
S(x_i, \varepsilon/4).$ We may assume that $\gamma(t)\not\in B(x_i,
\varepsilon/4)$ for $t>t_1.$ Put $\gamma_1:=\gamma|_{[t_1, 1]}.$
Similarly, $|\gamma_1|\cap B(x_i, \varepsilon/2)\ne\varnothing\ne
|\gamma_1|\cap (D\setminus B(x_i, \varepsilon/2)).$ By
Proposition~\ref{pr2} there is $t_1<t_2<1$ with $\gamma(t_2)\in
S(x_i, \varepsilon/2).$ We may assume that $\gamma(t)\in B(x_i,
\varepsilon/2)$ for $t<t_2.$ Put $\gamma_2:=\gamma|_{[t_1, t_2]}.$
Then, the path $\gamma_2$ is a subpath of $\gamma,$ which belongs to
the family $\Gamma(S(x_i, \varepsilon/4), S(x_i, \varepsilon/2),
A(x_i, \varepsilon/4, \varepsilon/2)).$ Thus, the
relation~(\ref{eq7C}) is established. By Proposition~\ref{pr3}
\begin{equation}\label{eq8C}
M(f_{m}(\Gamma(S(x_i, \varepsilon/4), S(x_i, \varepsilon/2)), A(x_i,
\varepsilon/4, \varepsilon/2)))\leqslant
\frac{\omega_{n-1}}{I_i^{n-1}}<\infty\,,
\end{equation}
where $I_i=I_i(x_i, \varepsilon/4,
\varepsilon/2)=\int\limits_{\varepsilon/4}^{\varepsilon/2}
\frac{dr}{rq_{x_i}^{\frac{1}{n-1}}(r)}$ and $I_i<\infty$ because
$q_{x_i}(r)<\infty$ for $r\in E_1$ and some set of positive linear
Lebesgue measure by the assumption of the lemma.

\medskip
Finally, by~(\ref{eq12}), (\ref{eq13}), (\ref{eq6D}), (\ref{eq7C})
and (\ref{eq8C}) we obtain that
$$M(\Gamma(f_m(A), E_m,\overline{{\Bbb R}^n}))=M(\Gamma_m)\leqslant$$
$$\leqslant M(f_m(\Delta_m))\leqslant \sum\limits_{i=1}^{M_0}M(f_m(\Gamma_{mi}))\leqslant$$
\begin{equation}\label{eq16}
\leqslant\sum\limits_{i=1}^{M_0}M(f_m(\Gamma(S(x_i, \varepsilon/4),
S(x_i, \varepsilon/2), A(x_i, \varepsilon/4,
\varepsilon/2))))\leqslant
\sum\limits_{i=1}^{M_0}\frac{\omega_{n-1}}{I_i^{n-1}}:=C<\infty\,.
\end{equation}
Since $P$ in~(\ref{eq9B}) may be done arbitrary big, the
relations~(\ref{eq9B}) and~(\ref{eq16}) contradict each other. This
completes the proof.~$\Box$
\end{proof}

\medskip
\begin{corollary}\label{cor4}
{\it\, The statement of Lemma~\ref{lem2} remains true if, instead of
the condition regarding the integrability of the function $Q$ over
spheres with respect to some set $E_1$ is replaced by a simpler
condition: $Q\in L^1(D).$}
\end{corollary}

\begin{proof}
It follows directly from Lemma~\ref{lem2} similarly to the proof of
Corollary~\ref{cor1}.~$\Box$
\end{proof}

\medskip
The following statement holds.

\medskip
\begin{proposition}\label{pr6}
{\it\, Let $D$ be a domain in ${\Bbb R}^n,$ $n\geqslant 2,$ and let
$x_0\in G,$ where $K:=\overline{G}$ is a compactum in a domain $D.$
Then $h(x_0,
\partial G)<h(x_0, \partial D).$}
\end{proposition}

\medskip
\begin{proof}
Since $\overline{{\Bbb R}^n}$ is connected, $\partial
D\ne\varnothing.$ By the same reason, $\partial G\ne\varnothing.$
Thus, $h(x_0,
\partial G)$ and $h(x_0,
\partial D)$ are well-defined.

Let us firstly prove that $B_h(z_0, h(z_0, \partial G))\subset G.$
Otherwise, $B_h(z_0, h(z_0, \partial G))\cap G\ne\varnothing\ne
B_h(z_0, h(z_0, \partial G))\setminus G.$ Since $B_h(z_0, h(z_0,
\partial G))$ is connected, by \cite[Theorem~1.I.5.46]{Ku$_2$} there
is $w_0\in B_h(z_0, h(z_0,
\partial G))\cap \partial G,$ so $h(z_0, \partial G)\leqslant h(w_0, z_0)<h(z_0, \partial G),$
which is impossible. Thus, $B_h(z_0, h(z_0, \partial G))\subset G,$
as required. Now, $\overline{B_h(z_0, h(z_0, \partial G))}\subset
\overline{G},$ as well.

Now, we prove that $h(x_0,
\partial G)<h(x_0, \partial D).$
Assume the contrary, i.e., $h(x_0,
\partial G)\geqslant h(x_0, \partial D).$ Since $\partial G$ and $\partial
D$ are compacts in $\overline{{\Bbb R}^n},$ there are $y_0\in
\partial G, z_0\in \partial D$ such that $h(x_0,
\partial G)=h(x_0, y_0)$ and $h(x_0,
\partial D)=h(x_0, z_0).$ By the assumption, $h(x_0, y_0)\geqslant h(x_0,
z_0).$ Now, by the proving above, $y_0\in \overline{B_h(x_0, h(x_0,
z_0))}\subset \overline{G}$ which is impossible, because $y_0\in
\partial D\cap K,$ but, at the same time, all of points $y_0$ in $K$ are inner with respect to
$D.$ The proposition is proved.~$\Box$
\end{proof}

\medskip
In Lemma~\ref{lem2} we have used the condition: ``$h(E_m)\geqslant
r$ whenever $E_m$ is a component of $\partial f_m(D)$''. Now we will
show that the statement of the lemma is true even without this
condition, but in the situation, when all $f_m(D)$ belong to some
fixed compact set $K\ne \overline{{\Bbb R}^n}.$

\medskip
Given a domain $D$ in ${\Bbb R}^n,$ $n\geqslant 2,$ a non-degenerate
continuum $A\subset D,$ a Lebesgue measurable function
$Q:D\rightarrow [0, \infty],$ a compactum $K\subset \overline{{\Bbb
R}^n},$ $K\ne \overline{{\Bbb R}^n},$ and a number $\delta>0$ we
denote by $\frak{B}^{\delta, K}_{A, Q}(D)$ a family of all
homeomorphisms $f:D\rightarrow K$
satisfying~(\ref{eq2*!A})--(\ref{eq8BC}) for all
$0<r_1<r_2<d_0:=\sup\limits_{x\in D}|x-x_0|$ such that
$h(f(A))\geqslant \delta.$ The following assertion is true.

\medskip
\begin{lemma}\label{lem4}{\it\, Let $D$ be a domain in ${\Bbb R}^n,$ $n\geqslant 2,$ let $A$
be a non-degenerate continuum in $D$ and let $f_m\in
\frak{B}^{\delta, K}_{A, Q}(D).$ Assume that, for each point $x_0\in
D$ and for every $0<r_1<r_2<r_0:=\sup\limits_{x\in D}|x-x_0|$ there
is a set $E_1\subset[r_1, r_2]$ of a positive linear Lebesgue
measure such that the function $Q$ is integrable with respect to
$\mathcal{H}^{n-1}$ over the spheres $S(x_0, r)$ for every $r\in
E_1.$ Then there exists $\delta_1>0$ such that $h(f_m(A),
\partial f_m(D))\geqslant \delta_1$ for any $m=1,2,\ldots .$ }
\end{lemma}

\medskip
\begin{proof} Let us prove that statement by the contradiction,
i.e., assume that there exists an increasing subsequence of numbers
$m_k,$ $k=1,2,\ldots ,$ such that $h(f_{m_k}(A),
\partial f_{m_k}(D))\rightarrow 0$ as $m\rightarrow\infty.$ Without loss of generality, we
may consider that the same sequence $f_m$ satisfy this condition,
i.e.
\begin{equation}\label{eq1}
h(f_m(A),
\partial f_m(D))\rightarrow 0\,,\qquad m\rightarrow\infty\,.
\end{equation}
Otherwise, we may consider the sequence $f_{m_k}$ instead of $f_m$
in the conditions of the lemma.

\medskip
Let $x_m, y_m$ be points in $A$ such that $h(f_m(x_m),
f_m(y_m))=h(f_m(A)) \geqslant \delta.$ We may assume that $x_m
\rightarrow x_0,$ $y_m\rightarrow y_0$ as $m\rightarrow\infty,$
$x_0, y_0\in A.$ Let $\varepsilon_0:=d(A, \partial D).$ If $D={\Bbb
R}^n,$ we set $\varepsilon_0=1.$ Besides that, let $z_m\in A$ be a
point such that
\begin{equation}\label{eq2A}
h(f_m(A), \partial f_m(D))=h(f_m(z_m), \zeta_m)\rightarrow
0\,,\qquad m\rightarrow\infty\,,
\end{equation}
for some $\zeta_m\in
\partial f_m(D).$ We also may consider that $z_m\rightarrow z_0\in
A$ as $m\rightarrow\infty,$
\begin{figure}[h]
\centerline{\includegraphics[scale=0.6]{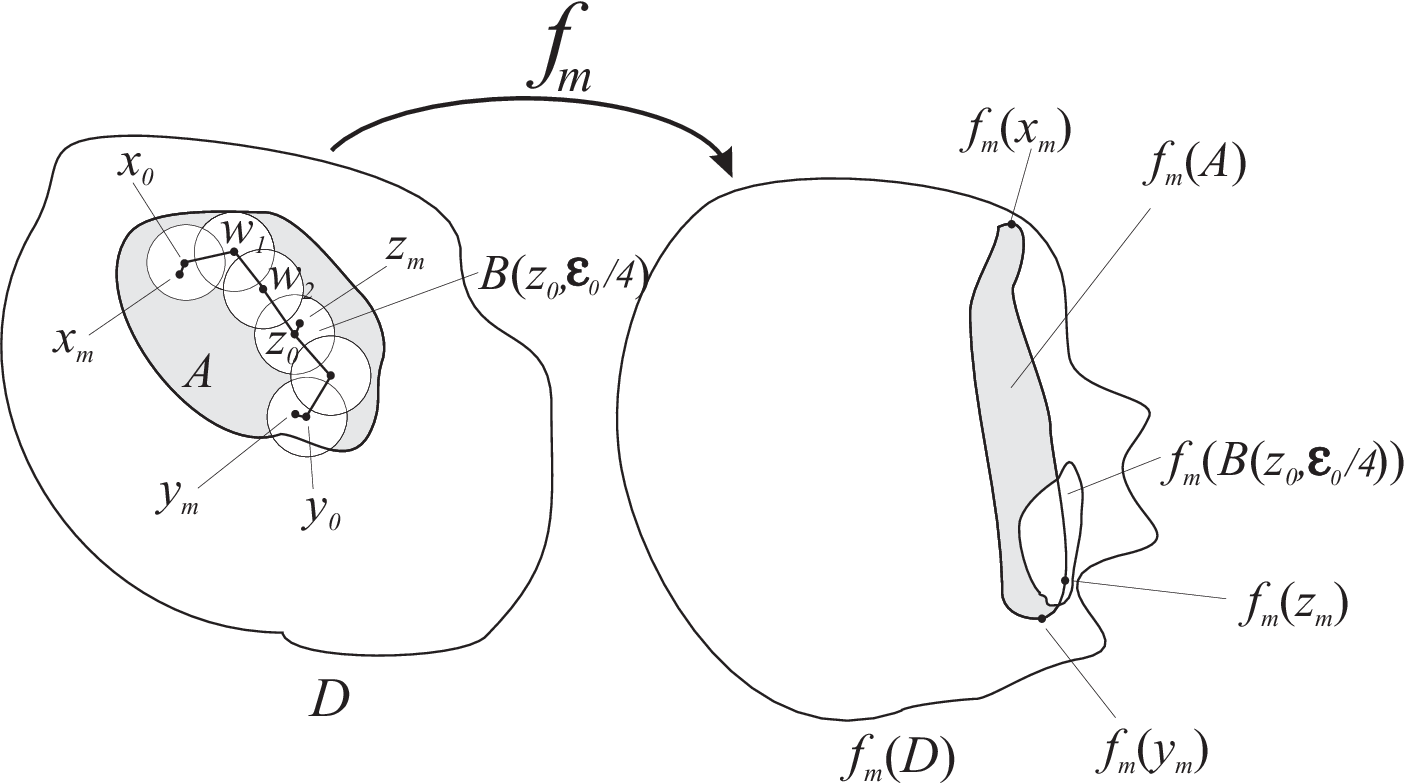}} \caption{To proof
of Lemma~\ref{lem4}}\label{fig1}
\end{figure}
and that $x_m\in B(x_0, \varepsilon_0/4),$ $y_m\in B(y_0,
\varepsilon_0/4)$ and $z_m\in B(z_0, \varepsilon_0/4)$ for all $m\in
{\Bbb N}$ (see Figure~\ref{fig1}). Now we sequently join the points
$x_0,$ $z_0$ and $y_0$ by the polygonal line connected the points
$x_0=w_0, w_1, w_2,\ldots w_k, w_{k+1}=z_0, w_{k+2}, w_{k+3},\ldots
, w_p=y_0$ such that $|w_{i}-w_{i+1}|<\varepsilon_0/4$ for any
$0\leqslant i\leqslant p-1.$ Observe that, by the triangle
inequality,
$$\delta\leqslant h(f_m(x_m), f_m(y_m))\leqslant h(f_m(x_m),
f_m(x_0))+\sum\limits_{i=1}^{k+1} h(f_m(w_{i-1}), f_m(w_i))+$$
\begin{equation}\label{eq5D}
+2h(f_m(z_0), f_m(z_m))+\sum\limits_{i=k+2}^{p} h(f_m(w_{i-1}),
f_m(w_i))+h(f_m(y_0), f_m(y_m))\,.
\end{equation}
Since the sum~(\ref{eq5D}) has only $p+4$ terms, there is at least
one term greater than or equal to $\delta_*:=\frac{\delta}{p+4}>0$
for an infinite number of values of $m.$ In what follows, our
reasoning should be divided into two cases:

\medskip
1) $h(f_m(z_0), f_m(z_m))\geqslant \delta_*>0$ for infinitely many
$m\in{\Bbb N},$

\medskip
2) $h(f_m(z_0), f_m(z_m))\rightarrow 0$ as $m\rightarrow\infty.$

\medskip
Let us first consider {\bf case 1)} (the simpler one). So, assume
that $h(f_m(z_0), f_m(z_m))\geqslant \delta_*>0$ for infinitely many
$m\in{\Bbb N}.$ Without loss of generality, using renumbering, if
required, we may assume that the latter holds for every $m\in {\Bbb
N}.$ In addition, since by the assumption $f_m(B(z_0,
\varepsilon_0/4))\subset K\ne \overline{{\Bbb R}^n},$ applying the
inversion $\psi(x)=\frac{x}{|x|^2}$ and considering mappings
$\widetilde{f}_m:=\psi\circ f_m,$ if required, we may consider that
$K$ is a compactum in ${\Bbb R}^n.$ Now, $f_m(\overline{B(z_0,
\varepsilon_0/2)})$ is a compactum in ${\Bbb R}^n$ and, obviously,
$$d(f_m(\overline{B(z_0, \varepsilon_0/2)}))=d(\partial(f_m(B(z_0,
\varepsilon_0/2))))=d(f_m(S(z_0, \varepsilon_0/2)))\,.$$
On the other hand, we set $A_1:=\overline{B(z_0, \varepsilon_0/4)}$
$$d(f_m(\overline{B(z_0,
\varepsilon_0/2)}))\geqslant d(f_m(A_1))\geqslant d(f_m(z_m),
f_m(z_0))\geqslant\delta_*>0$$
for all $m=1,2,\ldots ,$ because $z_m, z_0\in A_1$ for all $m\in
{\Bbb N}.$

Now, $d(\partial(f_m(B(z_0, \varepsilon_0/2))))\geqslant \delta_*>0$
for all $m=1,2,\ldots .$ Since the set $f_m(D)$ lies in the
compactum $K\subset {\Bbb R}^n,$ we obtain that
$h(\partial(f_m(B(z_0, \varepsilon_0/2))))\geqslant r_{**}>0$ for
all sufficiently large $m=1,2,\ldots $ and some $r_{**}>0$ while
$\partial(f_m(B(z_0, \varepsilon_0/2)))$ is connected as a image of
the connected set $S(z_0, \varepsilon_0/2)$ under the mapping $f_m.$
Now, $f_m\in \frak{A}^{\delta_*, r_{**}}_{A_1, Q}(B(z_0,
\varepsilon_0/2))$ and by Lemma~\ref{lem2} $h(f_m(A_1), \partial
f_m(B(z_0, \varepsilon_0/2)))\geqslant \delta_1>0$ for some
$\delta_1>0.$ Since $z_m\in A_1,$ it follows that $h(f_m(z_m),
\partial f_m(B(z_0, \varepsilon_0/2)))\geqslant \delta_1>0$ and,
consequently, by Proposition~\ref{pr6} $h(f_m(z_m), \partial
f_m(D))\geqslant \delta_1>0.$ But since $h(f_m(A), \partial
f_m(D))=h(f_m(z_m), \zeta_m)$ for some $\zeta_m\in \partial f_m(D),$
we have that $h(f_m(z_m), \zeta_m)\geqslant h(f_m(z_m), \partial
f_m(D))\geqslant \delta_1>0$ for sufficiently large $m\in {\Bbb N}.$
The latter contradicts with~(\ref{eq1}).

\medskip
Let us consider the {\bf case 2),} when $h(f_m(z_0),
f_m(z_m))\rightarrow 0$ as $m\rightarrow\infty.$ Here there are 2)
subcases:

\medskip
$2_1)$ $h(f_m(z_0), f_m(z_m))\rightarrow 0$ as $m\rightarrow\infty,$
$h(f_m(w_{i-1}), f_m(w_i))\rightarrow 0$ as $m\rightarrow\infty$ for
any $0\leqslant i\leqslant p,$ however, one of two following
conditions hold: either $h(f_m(x_m), f_m(x_0))\geqslant \delta_*>0$
or $h(f_m(y_m), f_m(y_0))\geqslant \delta_*>0$ for infinitely many
$m\in {\Bbb N};$

\medskip
$2_2)$ $h(f_m(z_0), f_m(z_m))\rightarrow 0$ as $m\rightarrow\infty,$
besides that, $h(f_m(w_{i-1}), f_m(w_i))\geqslant \delta_*>0$ for
infinitely many $m\in {\Bbb N}$ and at least for one $0\leqslant
i\leqslant p.$

\medskip
Let us consider {\bf the situation $2_1)$} firstly. Assume that
$h(f_m(x_m), f_m(x_0))\geqslant \delta_*>0$ for infinitely many
$m\in {\Bbb N}$ (the alternative situation, when $h(f_m(y_m),
f_m(y_0))\geqslant \delta_*>0,$ may be consider similarly). Now, we
have that $h(f_{m_k}(x_{m_k}), f_{m_k}(x_0))\geqslant \delta_*>0$
for all $k\in {\Bbb N}.$ We may consider that the latter holds for
$m$ instead $m_k,$ i.e., $h(f_m(x_m), f_m(x_0))\geqslant \delta_*>0$
for all $m\in {\Bbb N}.$ Otherwise, we consider $f_{m_k}$ instead of
$f_m$ in the conditions of the lemma.

Now we will actually repeat the reasoning from point~1). Indeed,
since by the assumption $f_m(B(x_0, \varepsilon_0/2))\subset K\ne
\overline{{\Bbb R}^n},$ applying the inversion
$\psi(x)=\frac{x}{|x|^2}$ and considering mappings
$\widetilde{f}_m:=\psi\circ f_m,$ if required, we may consider that
$K$ is a compactum in ${\Bbb R}^n.$ Now, $f_m(\overline{B(x_0,
\varepsilon_0/2)})$ is a compactum in ${\Bbb R}^n$ and, obviously,
$$d(f_m(\overline{B(x_0, \varepsilon_0/2)}))=d(\partial(f_m(B(x_0,
\varepsilon_0/2))))=d(f_m(S(x_0, \varepsilon_0/2)))\,.$$
On the other hand, we set $A_1:=\overline{B(x_0, \varepsilon_0/4)}$
$$d(f_m(\overline{B(x_0,
\varepsilon_0/2)}))\geqslant d(f_m(A_1))\geqslant d(f_m(x_m),
f_m(x_0))\delta_*>0$$
for all $m=1,2,\ldots ,$ because $x_m, x_0\in A_1$ for sufficiently
large $m\in {\Bbb N}.$

Now, $d(\partial(f_m(B(x_0, \varepsilon_0/2))))\geqslant \delta_*>0$
for all $m=1,2,\ldots .$ Since the set $f_m(D)$ lies in the
compactum $K\subset {\Bbb R}^n,$ we obtain that
$h(\partial(f_m(B(x_0, \varepsilon_0/2))))\geqslant r_{**}>0$ for
all sufficiently large $m=1,2,\ldots $ and some $r_{**}>0$ while
$\partial(f_m(B(x_0, \varepsilon_0/2)))$ is connected as a image of
the connected set $S(x_0, \varepsilon_0/2)$ under the mapping $f_m.$
Now, $f_m\in \frak{A}^{\delta_*, r_{**}}_{A_1, Q}(B(x_0,
\varepsilon_0/2))$ and by Lemma~\ref{lem2} $h(f_m(A_1), \partial
f_m(B(z_0, \varepsilon_0/2)))\geqslant \delta_1>0$ for some
$\delta_1>0.$ Since $x_m\in A_1,$ it follows that $h(f_m(x_m),
\partial f_m(B(x_0, \varepsilon_0/2)))\geqslant \delta_1>0$ and,
consequently, by Proposition~\ref{pr6} $h(f_m(x_m), \partial
f_m(D))\geqslant \delta_1>0.$ Observe that $h(f_m(x_m), \partial
f_m(D))\leqslant h(f_m(x_m), \zeta_m),$ where $\zeta_m\in \partial
f_m(D)$ is from~(\ref{eq2A}). Now
\begin{equation}\label{eq3}
h(f_m(x_m), \zeta_m)\geqslant \delta_1>0\,, m=1,2,\ldots .
\end{equation}
Now, since $h(f_m(w_{i-1}), f_m(w_i))\rightarrow 0$ as
$m\rightarrow\infty$ for any $0\leqslant i\leqslant p$ and
simultaneously $h(f_m(z_0), f_m(z_m))\rightarrow 0$ as
$m\rightarrow\infty,$ by the triangle inequality and~(\ref{eq3}) we
obtain that
$$h(f_m(z_m), \zeta_m)\geqslant h(\zeta_m, f_m(x_m))-h(f_m(x_m), f_m(z_m))\geqslant$$
\begin{equation}\label{eq4}
\geqslant h(\zeta_m, f_m(x_m))-\sum\limits_{i=1}^{k+1}
h(f_m(w_{i-1}), f_m(w_i))-h(f_m(z_0), f_m(z_m))\geqslant
\delta_1/2>0
\end{equation}
since the later two terms with a minus sign tend to zero by the
assumption in this subparagraph~$2_1).$ The relation~(\ref{eq4})
contradicts with~(\ref{eq2A}). The latter completes the
consideration of item~$2_1).$

\medskip
It remains to consider the situation~$2_2).$ Observe that the
consideration of this situation is no different from the previous
one, however, we will conduct this consideration in full. Indeed,
let $h(f_m(w_{i-1}), f_m(w_i))\geqslant \delta_*>0$ for infinitely
many $m\in {\Bbb N}$ and at least for one $0\leqslant i\leqslant p.$
As above, we may consider that latter holds for any $m\in {\Bbb N}.$
Now we consider the case $0\leqslant i\leqslant k+1,$ while the
consideration of the case $k+1\leqslant i\leqslant p$ may be done
similarly. We also may consider that $i$ be the largest index
between $0\leqslant i\leqslant p$ for which the relation
$h(f_m(w_{i-1}), f_m(w_i))\geqslant \delta_*>0$ holds for infinitely
many $m\in {\Bbb N}$ (otherwise, we may increase $i$). Now, by the
definition, $h(f_m(w_{i-1}), f_m(w_i))\geqslant \delta_*>0$ for
$m\in {\Bbb N},$ and $h(f_m(w_{l-1}), f_m(w_l))\rightarrow 0$ as
$m\rightarrow\infty$ for all $i<l\leqslant k+1.$

\medskip
By the assumption $f_m(B(w_{i}, \varepsilon_0/2))\subset K\ne
\overline{{\Bbb R}^n},$ applying the inversion
$\psi(x)=\frac{x}{|x|^2}$ and considering mappings
$\widetilde{f}_m:=\psi\circ f_m,$ if required, we may consider that
$K$ is a compactum in ${\Bbb R}^n.$ Now, $f_m(\overline{B(w_{i},
\varepsilon_0/2)})$ is a compactum in ${\Bbb R}^n$ and, obviously,
$$d(f_m(\overline{B(w_{i}, \varepsilon_0/2)}))=d(\partial(f_m(B(w_{i},
\varepsilon_0/2))))=d(f_m(S(w_{i}, \varepsilon_0/2)))\,.$$
On the other hand, we set $A_1:=\overline{B(x_0, \varepsilon_0/4)}$
$$d(f_m(\overline{B(w_{i},
\varepsilon_0/2)}))\geqslant d(f_m(A_1))\geqslant d(f_m(w_{i-1}),
f_m(w_{i}))\delta_*>0$$
for all $m=1,2,\ldots ,$ because $w_{i-1}, w_{i}\in A_1$ for every
$m\in {\Bbb N}.$

Now, $d(\partial(f_m(B(w_{i}, \varepsilon_0/2))))\geqslant
\delta_*>0$ for all $m=1,2,\ldots .$ Since the set $f_m(D)$ lies in
the compactum $K\subset {\Bbb R}^n,$ we obtain that
$h(\partial(f_m(B(w_{i}, \varepsilon_0/2))))\geqslant r_{**}>0$ for
all $m=1,2,\ldots $ and some $r_{**}>0$ while $\partial(f_m(B(w_{i},
\varepsilon_0/2)))$ is connected as a image of the connected set
$S(w_{i}, \varepsilon_0/2)$ under the mapping $f_m.$ Now, $f_m\in
\frak{A}^{\delta_*, r_{**}}_{A_1, Q}(B(w_{i}, \varepsilon_0/2))$ and
by Lemma~\ref{lem2} $h(f_m(A), \partial f_m(B(w_{i},
\varepsilon_0/2)))\geqslant \delta_1>0$ for some $\delta_1>0.$ Since
$w_{i}\in A_1,$ it follows that $h(f_m(w_{i}),
\partial f_m(B(w_{i}, \varepsilon_0/2)))\geqslant \delta_1>0$ and,
consequently, by Proposition~\ref{pr6} we have that $h(f_m(w_{i}),
\partial f_m(D))\geqslant \delta_1>0.$ Observe that $h(f_m(w_{i}),
\partial f_m(D))\leqslant h(f_m(w_{i}), \zeta_m),$ where $\zeta_m\in
\partial f_m(D)$ is from~(\ref{eq2A}). Now
\begin{equation}\label{eq3B}
h(f_m(w_{i}), \zeta_m)\geqslant \delta_1>0\,, m=1,2,\ldots .
\end{equation}
Now, since $h(f_m(w_{l-1}), f_m(w_l))\rightarrow 0$ as
$m\rightarrow\infty$ for any $i<l\leqslant k+1$ and simultaneously
$h(f_m(z_0), f_m(z_m))\rightarrow 0$ as $m\rightarrow\infty,$ by the
triangle inequality and~(\ref{eq3B}) we obtain that
$$h(f_m(z_m), \zeta_m)\geqslant h(\zeta_m, f_m(w_{i}))-h(f_m(w_{i}), f_m(z_m))\geqslant$$
\begin{equation}\label{eq4A}
\geqslant h(\zeta_m, f_m(w_{i}))-\sum\limits_{l=i+1}^{k+1}
h(f_m(w_{l-1}), f_m(w_l))-h(f_m(z_0), f_m(z_m))\geqslant
\delta_1/2>0
\end{equation}
since the latter two terms with a minus sign tend to zero by the
assumption in this subparagraph~$2_2).$ The relation~(\ref{eq4A})
contradicts with~(\ref{eq2A}). The latter completes the
consideration of item~$2_2).$ The proof is complete.~$\Box$
\end{proof}

\medskip
\begin{corollary}\label{cor5}
{\it\, The statement of Lemma~\ref{lem4} remains true if, instead of
the condition regarding the integrability of the function $Q$ over
spheres with respect to some set $E_1$ is replaced by a simpler
condition: $Q\in L^1(D).$}
\end{corollary}

\medskip
Using Lemma~\ref{lem2}, we obtain the following.

\medskip
\begin{proposition}\label{pr4}
{\it\, Assume that all the conditions of Lemma~\ref{lem2} hold. Then
for any continuum $A^{\,*}$ in $D$ there is $\delta^{\,*}_1>0$ such
that $h(f_m(A^{\,*}),
\partial f_m(D))\geqslant \delta^{\,*}_1$ for any $m=1,2,\ldots .$}
\end{proposition}

\medskip
\begin{proof}
Indeed, by Lemma~\ref{lem2} there exists $\delta_1>0$ such that
$h(f_m(A),
\partial f_m(D))\geqslant \delta_1$ for any $m=1,2,\ldots ,$ where
$A$ is a continuum from the conditions of the lemma. Assume the
contrary, then $h(f_{m_k}(A^{\,*}),
\partial f_{m_k}(D))\rightarrow 0$ for some increasing sequence of
numbers $m_k,$ $k=1,2,\ldots.$ Let us join the continua $A$ and
$A^{\,*}$ by a path $C$ in $D$ and consider the continuum $E:=A\cup
A^{\,*}\cup C.$ We have $h(f_{m_k}(E))\geqslant
h(f_{m_k}(A))\geqslant \delta$ for $k=1,2,\ldots ,$ so by
Lemma~\ref{lem2} there is $\delta_2>0$ such that $h(f_{m_k}(E),
\partial f_{m_k}(D))\geqslant \delta_2$ for any $k=1,2,\ldots .$
However, $h(f_{m_k}(E),
\partial f_{m_k}(D))\leqslant h(f_{m_k}(A^{\,*}),
\partial f_{m_k}(D))\rightarrow 0$ as $k\rightarrow\infty,$ because
$A^{\,*}\subset E.$ The obtained contradiction disproves the
assumption made above.~$\Box$
\end{proof}

\medskip
{\it Proof of Theorem~\ref{th1A}}. Let $C$ be a continuum in $D.$ By
Lemma~\ref{lem2} and Proposition~\ref{pr4}, there is
$\delta^{\,*}_1>0$ such that $h(f(C),
\partial f(D))\geqslant \delta^{\,*}_1$ for any $f\in\frak{A}^{\delta, r}_{A, Q}(D).$
Thus, $f\in\frak{A}^{\delta, r}_{A, Q}(D)\subset
\frak{F}^{\delta^{\,*}_1}_{C, Q}(D).$ In this case, the desired
statement follows directly from Theorems~\ref{th1} and
\ref{th2}.~$\Box$

\medskip
{\it Proof of Corollary~\ref{cor3}} is based on Theorem~\ref{th1A}
and is similar to the proof of Corollary~\ref{cor1}.~$\Box$

\medskip
\begin{example}
Let $D$ be a any proper subdomain of ${\Bbb R}^n,$ $n\geqslant 2.$
Due to the connectedness of ${\Bbb R}^n$ there is a point $y\in
\partial D.$ Given $m\in {\Bbb N},$ we set $f_m(x)=mx.$ Let $K\subset D,$
$K\ne \{0\},$ be a compact set. Let $x\in K,$ then
$f_m(x)\rightarrow\infty$ as $m\rightarrow\infty$ and also
$f_m(y)\rightarrow\infty$ as $m\rightarrow\infty.$ Now, $h(f_m(K),
\partial f_m(D))\rightarrow 0$ as $m\rightarrow\infty.$ Thus, the
statements of Lemmas~\ref{lem2} and~\ref{lem4} do not hold for
$f_m,$ $m=1,2,\,.$ Since the mappings $f_m$ are conformal for any
$m=1,2,\ldots ,$ they are ring $Q$-mappings for $Q\equiv 1$ (see,
e.g., \cite[Theorem~8.1]{Va}). If $D$ is bounded, then $Q\in
L^1(D),$ in particular, $Q$ is integrable on all spheres. Therefore,
the reason that the conclusions of these lemmas are not true is that
in Lemma~\ref{lem2}, $f_m\not\in \frak{A}^{\delta, r}_{A, Q}(D),$
because the condition ``$h(E)\geqslant r$ whenever $E$ is any
component of $\partial f_m(D)$'' is violated. Besides that, in
Lemma~\ref{lem4}, $f_m\not\in \frak{B}^{\delta, K_1}_{A, Q}(D),$
because given a compactum $K_1\subset \overline{{\Bbb R}^n},$
$K_1\ne \overline{{\Bbb R}^n},$ we have that $f_m(x)\not\in K_1$ for
some $x\in D$ and some $m\in {\Bbb N}.$ Observe that, the condition
$h(f_m(A))\geqslant \delta>0$ holds for any $m\in {\Bbb N}$ and some
a continuum $A$ in $D;$ for example, the latter hold if $A$ is a
closed unit ball centered at the origin.

\medskip
Observe also that the family $\{f_m\}_{m=1}^{\infty}$ is neither
uniformly open nor uniformly light. In particular, if $B$ is a ball
that does not contain the origin, then $h(f_m(B))\rightarrow 0$ as
$m\rightarrow\infty,$ because $f_m$ converges to $\infty$ uniformly
on $B.$ Finally, if $f_m(B(z_0, \varepsilon_1))\supset B_h(f_m(z_0),
r_0)$ for some $z_0\in {\Bbb R}^n$ and some $\varepsilon_1, r_0>0$
and all $m\in {\Bbb N},$ then $h(f_m(B(z_0,
\varepsilon_1)))\geqslant h(B_h(f_m(z_0), r_0))\geqslant r_0>0,$ but
this is not true due to the mentioned above.
\end{example}

\medskip
\begin{example}
Now we consider the example of mappings
from~\cite[Example~5.3]{SSD}. Let $p\geqslant 1,$ $n/p(n-1)<1.$ Put
$\alpha\in (0, n/p(n-1)).$ Let $f_m: B(0, 2)\rightarrow {\Bbb B}^n,$
$$f_m(x)\,=\,\left
\{\begin{array}{rr} \frac{(|x|-1)^{1/\alpha}}{|x|}\cdot x\,,
& 1+1/(m^{\alpha})\leqslant|x|\leqslant 2, \\
\frac{1/m}{1+(1/m)^{\alpha}}\cdot x\,, & 0<|x|<1+1/(m^{\alpha}) \ .
\end{array}\right.
$$
Using the approach applied in~\cite[Proposition~6.3]{MRSY} and
applying \cite[Theorems~8.1, 8.5]{MRSY} we may show that $f_m$
satisfy the relations~(\ref{eq2*!A})--(\ref{eq8BC}) in $B(0, 2)$
with $Q(x)=\frac{1}{\alpha}\cdot\frac{|x|}{(|x|-1)}$ for $x\in B(0,
2)\setminus \overline{{\Bbb B}^n},$ $Q(x)=1$ for $x\in {\Bbb B}^n.$
Observe that $f_m$ are homeomorphisms for all $m=1,2,\ldots,$ such
that $f_m(x)\rightarrow f(x)$ locally uniformly in $B(0, 2),$ where
$$f(x)\,=\,\left
\{\begin{array}{rr} \frac{(|x|-1)^{1/\alpha}}{|x|}\cdot x\,,
& 1\leqslant|x|\leqslant 2, \\
0\,, & 0<|x|<1 \ .
\end{array}\right.
$$
Obviously, the family $\{f_m\}_{m=1}^{\infty}$ is neither uniformly
open nor uniformly light in $B(0, 2),$ and the limit mapping of the
sequence $f_m,$ $m=1,2,\ldots ,$ is neither a homeomorphism nor a
constant. The reason that none of the conclusions of
Theorems~\ref{th1}--\ref{th1A} are satisfied is that the function
$Q$ is not integrable in the domain $D=B(0, 2),$ and moreover, is
not integrable at some family of spheres.  If we consider the domain
$B(0, 2)\setminus\{0\}$ as the domain $D,$ then we may see that the
statements of Lemmas~\ref{lem2} and~\ref{lem4} are also not
satisfied, since any continuum $K_1$, lying in $B(0, 1)\setminus
\{0\}$ converges to the origin under the mappings $f_m$, while $0\in
\partial f_m(D)=\partial (B(0, 1)\setminus\{0\}).$
\end{example}

\medskip
\begin{proposition}\label{pr5}
{\it\, Assume that all the conditions of Lemma~\ref{lem4} hold. Then
for any continuum $A^{\,*}$ in $D$ there is $\delta^{\,*}_1>0$ such
that $h(f_m(A^{\,*}),
\partial f_m(D))\geqslant \delta^{\,*}_1$ for any $m=1,2,\ldots .$}
\end{proposition}

\medskip
\begin{proof}
Indeed, by Lemma~\ref{lem4} there exists $\delta_1>0$ such that
$h(f_m(A),
\partial f_m(D))\geqslant \delta_1$ for any $m=1,2,\ldots ,$ where
$A$ is a continuum from the conditions of the lemma. Assume the
contrary, then $h(f_{m_k}(A^{\,*}),
\partial f_{m_k}(D))\rightarrow 0$ for some increasing sequence of
numbers $m_k,$ $k=1,2,\ldots.$ Let us join the continua $A$ and
$A^{\,*}$ by a path $C$ in $D$ and consider the continuum $E:=A\cup
A^{\,*}\cup C.$ We have $h(f_{m_k}(E))\geqslant
h(f_{m_k}(A))\geqslant \delta$ for $k=1,2,\ldots ,$ so by
Lemma~\ref{lem4} there is $\delta_2>0$ such that $h(f_{m_k}(E),
\partial f_{m_k}(D))\geqslant \delta_2$ for any $k=1,2,\ldots .$
However, $h(f_{m_k}(E),
\partial f_{m_k}(D))\leqslant h(f_{m_k}(A^{\,*}),
\partial f_{m_k}(D))\rightarrow 0$ as $k\rightarrow\infty,$ because
$A^{\,*}\subset E.$ The obtained contradiction disproves the
assumption made above.~$\Box$
\end{proof}

\medskip
The following statement holds.

\medskip
\begin{theorem}\label{th2A}
{\it\,Assume that, for each point $x_0\in D$ and for every
$0<r_1<r_2<r_0:=\sup\limits_{x\in D}|x-x_0|$ there is a set
$E_1\subset[r_1, r_2]$ of a positive linear Lebesgue measure such
that the function $Q$ is integrable with respect to
$\mathcal{H}^{n-1}$ over the spheres $S(x_0, r)$ for every $r\in
E_1.$

\medskip
1) If $D$ is bounded, then the family $\frak{B}^{\delta, K}_{A,
Q}(D)$ is equi-uniform (uniformly light) and uniformly open on any
compactum $K_1\subset D;$

\medskip
2) if $Q\in L^1(D)$ and $f(x)\ne\infty$ for every $x\in D,$ then for
any $K_1\subset D$ there exist constants $C, C_1>0$ such that
$$
|f(x)-f(y)|\geqslant C_1\cdot\exp\left\{-\frac{\Vert
Q\Vert_1}{C|x-y|^n}\right\}
$$
for all $x, y\in K_1$ and every $f\in \frak{A}^{\delta, r}_{A,
Q}(D).$ }
\end{theorem}

\medskip
\begin{proof}
By Lemma~\ref{lem4} and Proposition~\ref{pr5}, $\frak{B}^{\delta,
K}_{A, Q}(D)\subset \frak{F}^{\delta_1}_{K_1, Q}(D)$ for some
$\delta_1>0.$ The desired conclusion follows now by
Theorems~\ref{th1} and~\ref{th2}.~$\Box$
\end{proof}

\medskip
Arguing similarly to the proof of Corollary~\ref{cor1}, by
Theorem~\ref{th2A} we obtain the following.

\medskip
\begin{corollary}\label{cor6}
{\it\, The statement of Theorem~\ref{th2A} remains true if, instead
of the condition regarding the integrability of the function $Q$
over spheres with respect to some set $E_1$ is replaced by a simpler
condition: $Q\in L^1(D).$}
\end{corollary}

\medskip
\centerline{\bf 3. Convergence theorems.}

\medskip
{\it Proof of Theorem~\ref{th4}} Let us prove that $f$ is discrete.
Assume the contrary. Then there is $x_0 \in D$ and a sequence
$x_m\in D,$ $m=1,2,\ldots ,$ $x_m \ne x_0,$ such that
$x_m\rightarrow x_0$ as $m\rightarrow \infty$ and $f(x_m)=f(x_0).$
Observe that $E_0=\{x\in D: f(x)=f(x_0)\}$ is closed in $D$ by the
continuity of $f$ and does not coincide with $D,$ because $f\not
\equiv const.$ Thus, we may assume that $x_0$ may be replaced by non
isolated boundary point of $E_0.$

\medskip
Now, there are $\varepsilon_0>0,$ $0<\varepsilon_0<
\frac{1}{2}\cdot{\rm dist\,}(x_0,
\partial D),$ and $a, b\in B(x_0, \varepsilon_0),$
$a\ne b,$ such that $h(f(a), f(b)):=\delta>0.$ Since $f_m$ converges
to $f$ locally uniformly, $h(f_m(a), f_m(b))\geqslant \delta/2>0$
for sufficiently large $m\in {\Bbb N}.$ Observe that, the case
$0<r_1<r_2<{\rm dist\,}(x_0,
\partial D)$ includes the case $0<r_1<r_2<2\varepsilon_0=d(B(x_0, \varepsilon_0)).$
Let $A:=|\gamma|,$ where $\gamma$ is a path joining the points $a$
and $b$ in $B(x_0, \varepsilon_0).$ Now, $h(f_m(A))\geqslant
\delta>0$ for all $m=1,2,\ldots.$

\medskip
Since $f$ is continuous, the sets $f_m(B(x_0, \varepsilon_0))$ lie
in some neighborhood $U=f(B(x_0, \varepsilon_0))$ of a point $x_0.$
Increasing $\varepsilon_0,$ if required, we may assume that
$U\subset \overline{B_h(f(x_0), 1/2)}\ne \overline{{\Bbb R}^n}.$
Now, $f_m\in \frak{B}^{\delta, K}_{A, Q}(D)$ for
$K=\overline{B_h(f(x_0), 1/2)}.$ Using the inversion
$\psi(x)=\frac{x}{|x|^2}$ and considering mappings
$\widetilde{f}_m:=\psi\circ f_m,$ if required, we may consider that
$K\subset {\Bbb R}^n.$ Now, by Theorem~\ref{th2A} and
Corollary~\ref{cor6} there exist constants $C, C_1>0$ such that
\begin{equation}\label{eq1B}
|f_m(x)-f_m(y)|\geqslant C_1\cdot\exp\left\{-\frac{\Vert
Q\Vert_1}{C|x-y|^n}\right\}
\end{equation}
for all $x, y\in K,$ where $\Vert Q\Vert_1$ denotes the $L^1$-norm
of $Q$ in $B(x_0, \varepsilon_0).$

Let us prove the discreteness $f$ by the contradiction. Indeed, $f$
is not discrete, then there is a point $x_0\in D$ and a sequence
$x_k\in G,$ $x_k \ne x_0,$ $k=1,2\ldots,$ such that $x_k \rightarrow
x_0$ as $k\rightarrow \infty$ with $f(x_k)=f(x_0).$ Passing to the
limit as $m\rightarrow\infty$ in~(\ref{eq1B}), we obtain that
$$|f(x)-f(x_0)|\geqslant C_1
\exp\left\{-\frac{2\Vert Q\Vert_1}{C|x-x_0|^n}\right\}\,.$$
In particular,
$$0=|f(x_k)-f(x_0)|\geqslant
C_1\exp\left\{-\frac{2\Vert Q\Vert_1}{C|x_k-x_0|^n}\right\}>0\,.$$
The above contradiction proves the discreteness of $f.$ Finally, $f$
is a homeomorphism due to~\cite[Theorem~3.1]{RS}.

\medskip
Let now $f_m(x)\ne \infty$ for all $x\in D.$ It remains to show that
$f$ either a homeomorphism $f:D\rightarrow {\Bbb R}^n,$ or a
constant $c\in \overline{{\Bbb R}^n}.$ Assume that, $f$ is not a
constant $c\in \overline{{\Bbb R}^n}.$ We need to show that $f$ is a
homeomorphism $f:D\rightarrow {\Bbb R}^n.$

Let $E_0$ be a set of points $x$ in $D$ where $f(x)=\infty.$ Now,
$E_0$ is closed in $D$ and $E_0\ne D.$ We show that
$E_0=\varnothing.$ In the contrary case, there is $x_0\in D\cap
\partial E_0.$ Let $\varepsilon_1>0$ be such that $\overline{B(x_0,
\varepsilon_1)}\subset D.$ Arguing as above, we may show that
$f_m\in \frak{B}^{\delta, K}_{A, Q}(D)$ for $K=\overline{B_h(f(x_0),
1/2)},$ some a continuum $A$ and some $\delta>0.$ Now, by
Theorem~\ref{th2A} the family $f_m,$ $m=1,2,\ldots ,$ is uniformly
open, that is there is $r_*>0,$ which does not depend on $m,$ such
that $B_h(f_m(x_0), r_*)\subset f_m(B(x_0, \varepsilon_1)),$
$m=1,2,\ldots,$ for some $r_*>0.$ Let $y\in B_h(f(x_0),
r_*/2)=B_h(f(x_0), r_*/2).$ By the convergence of $f_m$ to $f$ and
by the triangle inequality, we obtain that
$$h(y, f_m(x_0))\leqslant h(y, f(x_0))+h(f(x_0), f_m(x_0))<r_*/2+r_*/2=r_*$$
for sufficiently large $m\in {\Bbb N}.$ Thus,
$$B_h(f(x_0), r_*/2)\subset B_h(f_m(x_0), r_*)\subset f_m(B(x_0,
\varepsilon_1))\subset {\Bbb R}^n\,.$$
In particular, $y_0=f(x_0)\in {\Bbb R}^n,$ as required. However,
$f(x_0)=\infty$ because $x_0\in \partial E_0$ and $E_0$ is closed.
Thus, $E_0=\varnothing,$ as required. Theorem~\ref{th4} is
proved.~$\Box$

\medskip
Given a domain $D$ in ${\Bbb R}^n,$ $n\geqslant 2,$ $\delta>0$ and
$a, b\in D$ we denote by $\frak{K}^{\delta}_{a, b, Q}$ a class of
all homeomorphisms $f:D\rightarrow \overline{{\Bbb R}^n}$ satisfying
(\ref{eq2*!A})--(\ref{eq8BC}) at every point $x_0\in D$ for every
$0<r_1<r_2<r_0:={\rm dist\,}(x_0, \partial D)$ such that $h(f(a),
f(b))\geqslant \delta>0.$

\medskip
\begin{theorem}\label{th5}
Let $Q\in L_{\rm loc}^1(D).$ Then the family $\frak{K}^{\delta}_{a,
b, Q}$ is closed in the topology of the locally uniform convergence.
In other words, if $f_m\in \frak{K}^{\delta}_{a, b, Q}$ converges to
some $f:D\rightarrow\overline{{\Bbb R}^n}$ locally uniformly, then
$f$ is a homeomorphism $f:D\rightarrow \overline{{\Bbb R}^n}$
satisfying the relations (\ref{eq2*!A})--(\ref{eq8BC}) at every
point $x_0\in D$ for every $0<r_1<r_2<r_0:={\rm dist\,}(x_0,
\partial D);$ besides that, $h(f(a), f(b))\geqslant \delta>0.$
\end{theorem}

\medskip
\begin{proof}
Under the notions of Theorem~\ref{th5}, $f$ is a homeomorphism by
Theorem~\ref{th4}. In addition, letting into the limit in the
condition $h(f_m(a), f_m(b))\geqslant \delta>0$ as
$m\rightarrow\infty,$ we obtain $h(f(a), f(b))\geqslant \delta>0,$
as required. Finally, $f$ is a ring $Q$-homeomorphism by Theorem~5.1
in \cite{RS}.~$\Box$
\end{proof}

\medskip
{\it Proof of Theorem~\ref{th6}}. Let us prove that $f$ is discrete.
Assume the contrary. Then there is $x_0 \in D$ and a sequence
$x_k\in D,$ $k=1,2,\ldots ,$ $x_k \ne x_0,$ such that
$x_k\rightarrow x_0$ as $k\rightarrow \infty$ and $f(x_k)=f(x_0).$
Observe that $E_0=\{x\in D: f(x)=f(x_0)\}$ is closed in $D$ by the
continuity of $f$ and does not coincide with $D,$ because $f\not
\equiv const.$ Thus, we may assume that $x_0$ may be replaced by non
isolated boundary point of $E_0.$

\medskip
Now, there are $\varepsilon_0>0,$ $0<\varepsilon_0<
\frac{1}{2}\cdot{\rm dist\,}(x_0,
\partial D),$ and $a, b\in B(x_0, \varepsilon_0),$
$a\ne b,$ such that $h(f(a), f(b)):=\delta>0.$ Since $f_m$ converges
to $f$ locally uniformly, $h(f_m(a), f_m(b))\geqslant \delta/2>0$
for sufficiently large $m\in {\Bbb N}.$ Observe that, the case
$0<r_1<r_2<{\rm dist\,}(x_0,
\partial D)$ includes the case $0<r_1<r_2<2\varepsilon_0=d(B(x_0, \varepsilon_0)).$
Let $A:=|\gamma|,$ where $\gamma$ is a path joining the points $a$
and $b$ in $B(x_0, \varepsilon_0).$ Now, $h(f_m(A))\geqslant
\delta>0$ for all $m=1,2,\ldots.$

\medskip
Since $f$ is continuous, the sets $f_m(B(x_0, \varepsilon_0))$ lie
in some neighborhood $U=f(B(x_0, \varepsilon_0))$ of a point $x_0.$
Increasing $\varepsilon_0,$ if required, we may assume that
$U\subset \overline{B_h(f(x_0), 1/2)}\ne \overline{{\Bbb R}^n}.$
Now, $f_m\in \frak{B}^{\delta, K}_{A, Q}(D)$ for
$K=\overline{B_h(f(x_0), 1/2)}.$ Using the inversion
$\psi(x)=\frac{x}{|x|^2}$ and considering mappings
$\widetilde{f}_m:=\psi\circ f_m,$ if required, we may consider that
$K\subset {\Bbb R}^n.$ Now, by Theorem~\ref{th2A} the family
$\{f_m\}_{m=1}^{\infty}$ is equi-uniform on any compactum $K_1.$ We
may set $K_1:=\overline{B(x_0, \varepsilon_0/2)},$ now $x_k, y_k\in
K_1$ for sufficiently large $k=1,2,\ldots .$

\medskip
Now, by the triangle inequality,
$$|f(x_k)-f(x_0)|=|f(x_k)-f_k(x_k)+f_m(x_k)-f_m(x_0)+f_k(x_0)-f(x_0)|\geqslant$$
\begin{equation}\label{eq1E}
\geqslant |f_m(x_k)-f_m(x_0)|-|f(x_k)-f_m(x_k)|-|f_m(x_0)-f(x_0)|\,.
\end{equation}
Since $|x_k-x_0|:=\varepsilon_k>0$ and the family
$\{f_m\}_{m=1}^{\infty}$ is equi-uniform, it follows that
$|f_m(x_k)-f_m(x_0)|\geqslant \delta_k>0$ for some $\delta_k>0.$ Fix
$k\in {\Bbb N}$ and letting in~(\ref{eq1E}) to the limit as
$m\rightarrow\infty,$ we obtain that
\begin{equation}\label{eq1C}
|f(x_k)-f(x_0)|\geqslant |f_m(x_k)-f_m(x_0)|\geqslant \delta_k>0\,.
\end{equation}
The relation~(\ref{eq1C}) contradicts the assumption $f(x_k)=f(x_0)$
for sufficiently large $k\in {\Bbb N}.$ Thus, $f$ is a discrete, as
required. Finally, $f$ is a homeomorphism due
to~\cite[Theorem~3.1]{RS}.

\medskip
Let now $f_m(x)\ne \infty$ for all $x\in D.$ It remains to show that
$f$ either a homeomorphism $f:D\rightarrow {\Bbb R}^n,$ or a
constant $c\in \overline{{\Bbb R}^n}.$ Assume that, $f$ is not a
constant $c\in \overline{{\Bbb R}^n}.$ We need to show that $f$ is a
homeomorphism $f:D\rightarrow {\Bbb R}^n.$

Let $E_0$ be a set of points $x$ in $D$ where $f(x)=\infty.$ Now,
$E_0$ is closed in $D$ and $E_0\ne D.$ We show that
$E_0=\varnothing.$ In the contrary case, there is $x_0\in D\cap
\partial E_0.$ Let $\varepsilon_1>0$ be such that $\overline{B(x_0,
\varepsilon_1)}\subset D.$ Arguing as above, we may show that
Arguing as above, we may show that $f_m\in \frak{B}^{\delta, K}_{A,
Q}(D)$ for $K=\overline{B_h(f(x_0), 1/2)},$ some a continuum $A$ and
some $\delta>0.$ Now, by Theorem~\ref{th2A} the family $f_m,$
$m=1,2,\ldots ,$ is uniformly open, that is there is $r_*>0,$ which
does not depend on $m,$ such that $B_h(f_m(x_0), r_*)\subset
f_m(B(x_0, \varepsilon_1)),$ $m=1,2,\ldots,$ for some $r_*>0.$ Let
$y\in B_h(f(x_0), r_*/2)=B_h(f(x_0), r_*/2).$ By the convergence of
$f_m$ to $f$ and by the triangle inequality, we obtain that
$$h(y, f_m(x_0))\leqslant h(y, f(x_0))+h(f(x_0), f_m(x_0))<r_*/2+r_*/2=r_*$$
for sufficiently large $m\in {\Bbb N}.$ Thus,
$$B_h(f(x_0), r_*/2)\subset B_h(f_m(x_0), r_*)\subset f_m(B(x_0,
\varepsilon_1))\subset {\Bbb R}^n\,.$$
In particular, $y_0=f(x_0)\in {\Bbb R}^n,$ as required. However,
$f(x_0)=\infty$ because $x_0\in \partial E_0$ and $E_0$ is closed.
Thus, $E_0=\varnothing,$ as required. Theorem~\ref{th6} is proved.

\medskip
{\bf Declarations.}

\medskip
{\bf Funding.}  No funds, grants, or other support was received.

\medskip
{\bf Conflicts of interest.} The author has no financial or
proprietary interests in any material discussed in this article.

\medskip
{\bf Availability of data and material.} The datasets generated
and/or analysed during the current study are available from the
corresponding author on reasonable request.


{\footnotesize

\begin{thebibliography}{99}

\bibitem{Cr} M.~Cristea, The limit mapping of generalized ring
homeomorphisms, {\it Complex Variables and Elliptic Equations,} {\bf
61} (2016), 608--622.

\bibitem{Fu} B.~Fuglede, Extremal length and functional
completion, {\it Acta Math.,} {\bf 98} (1957), 171--219.

\bibitem{GM} F.W.~Gehring and O.~Martio, Quasiextremal distance domains
and extension of quasiconformal mappings, {\it J. Anal. Math.,}
\textbf{45} (1985), 181-206.

\bibitem{He} J.~Heinonen, Lectures on Analysis on metric
spaces, Springer Science+Business Media (New York, 2001).

\bibitem{IRS} E.~Sevost'yanov, D.~Romash, N.~Ilkevych, On lower distance estimates of mappings in metric
spaces, https://arxiv.org/pdf/2411.03775 .

\bibitem{Ku$_2$} K.~Kuratowski, Topology, v.~2, Academic
Press (New York--London, 1968).

\bibitem{MRSY} O.~Martio, V. Ryazanov, U. Srebro, and E. Yakubov,
Moduli in modern mapping theory, Springer Science + Business Media,
LLC (New York, 2009).

\bibitem{Ri} S.~Rickman, Quasiregular mappings, Springer-Verlag (Berlin, 1993).

\bibitem{RS}  V.~Ryazanov, E.~Sevost'yanov,
On convergence and compactness of spatial homeomorphisms, {\it
Romanian Journal of Pure and Applies Mathematics}, \textbf{18}
(2013), 85--104.

\bibitem{Sa} S.~Saks, Theory of the Integral, Dover Publ. Inc. (New York, 1964).

\bibitem{SevSkv$_1$} E.~Sevost'yanov, S.~Skvortsov,
On mappings whose inverse satisfy the Poletsky inequality, {\it Ann.
Acad. Scie. Fenn. Math.,} \textbf{45} (2020), 259--277.

\bibitem{SevSkv$_2$} E.~Sevost'yanov, S.~Skvortsov,
Logarithmic H\"{o}lder continuous mappings and Beltrami equation,
{\it Analysis and Mathematical Physics,} \textbf{11} (2021), 138.

\bibitem{SSD} E.A.~Sevost’yanov, S.O.~Skvortsov and
O.P.~Dovhopiatyi, On non-homeomorphic mappings with inverse Poletsky
inequality, {\it Journal of Mathematical Sciences,} {\bf 252}
(2021), 541-557.

\bibitem{ST} E. Sevost'yanov, V. Targonskii. An analogue of Koebe's theorem and the openness of a
limit map in one class, {\it Analysis and Mathematical Physics,}
\textbf{15} (2025), Article number~59.

\bibitem{Va} J.~V\"{a}is\"{a}l\"{a}, {\it Lectures on $n$-dimensional quasiconformal
mappings}, Lecture Notes in Math. {\bf 229}, Springer-Verlag (Berlin
etc., 1971).

\bibitem{Vu} M.~Vuorinen, On the existence of angular limits of
$n$-dimensional quasiconformal mappings, {\it Ark. Math.,}
\textbf{18} (1980), 157--180.

\end{thebibliography}}
\end{document}